\documentclass[11pt]{article}
\usepackage{latexsym,amsfonts}
\usepackage{amssymb,amsthm,upref,amscd}
\usepackage[T1]{fontenc}
\usepackage{times}
\usepackage{amscd}
\usepackage{mathrsfs}
\usepackage[reqno]{amsmath}
\usepackage{constants}
\usepackage{booktabs}
\usepackage{verbatim}
\usepackage{tikz}
\usepackage{pstricks,pst-node,multido,ifthen,calc}
\usepackage{array}
 \usepackage{authblk}
\usepackage[numbers,sort&compress]{natbib}
 \usepackage{tocbibind}
\usepackage{appendix}

 \usetikzlibrary{decorations.markings}
\tikzstyle{vertex}=[circle, inner sep=1pt, minimum size=4pt]
\tikzstyle{directed}=[postaction={decorate,
                    decoration={markings,mark=at position 1 with {\arrow{latex}}}
                   }]
\newcommand{\vertex}{\node[vertex]}

\usepackage[colorlinks,linkcolor=red,citecolor=blue]{hyperref}
\usepackage{geometry}
\usepackage{graphicx}
\usepackage{subfigure}
\newtheorem{Theorem}{Theorem}[section]
\newtheorem{definition}[Theorem]{Definition}
\newtheorem{lemma}[Theorem]{Lemma}
\newtheorem{proposition}[Theorem]{Proposition}

\newtheorem{Open Problem}[Theorem]{Open Problem}

\makeatletter \@addtoreset{equation}{section} \makeatother

\newcommand{\abs}[1]{\lvert#1\rvert}

\newcommand{\st}{\,\vert\,}
\newcommand{\dif}{\,\mathrm{d}}
\newcommand{\Rset}{\mathbb{R}}

\newcommand{\Nset}{\mathbb{N}}
\newcommand{\wto}{\rightharpoonup}
\newcommand{\cE}{\mathcal{E}}
\newcommand{\cP}{\mathcal{P}}
\newcommand{\cH}{\mathcal{H}}
\newcommand{\cM}{\mathcal{M}}
\newcommand{\cN}{\mathcal{N}}
\newcommand{\cS}{\mathcal{S}}

\newcommand{\cF}{\mathcal{F}}
\DeclareMathOperator{\supp}{supp}
\DeclareMathOperator{\dist}{dist}

\begin{document}

\title{\bf Saddle solutions for the Choquard equation with a general nonlinearity
}

\author[]{Jiankang Xia\thanks{E-mail address: jiankangxia@nwpu.edu.cn.}}
\affil[]{{\footnotesize School of Mathematics and Statistics, Northwestern Polytechnical University,
           Xi'an 710129, China}}
\date{}
\maketitle

\begin{minipage}{13cm}
{\small {\bf Abstract:} In the spirit of Berestycki and Lions, we prove the existence of saddle type nodal solutions for the Choquard equation  
\[
 -\Delta u + u= \big(I_\alpha \ast F(u)\big)F'(u)\qquad \text{ in }\;\mathbb{R}^N
\]
where $N\geq 2$ and $I_\alpha$ is the Riesz potential of order $\alpha\in (0,N)$. Without any compact setting, we construct saddle solutions in a unified way for the  Choquard equation whose nodal domains are of cone shapes demonstrating Coxeter's symmetric configurations in $\Rset^N$.
Moreover, if $F'$ is odd and has constant sign on $(0,+\infty)$, then the saddle solution maintains signed on the fundamental domain of the corresponding Coxeter group and receives opposite signs on any two adjacent domains.
These results further complete the variational framework in constructing sign-changing solutions for the Choquard equation but still require a quadratic or super-quadratic growth on $F$ near the origin.


}
\end{minipage}

\section{Introduction and main results}
\label{section1}

We consider nodal solutions for the  Choquard type nonlinear equation
\begin{equation}
\tag{$\mathcal{C}$}
\label{eqChoquard}
- \Delta u+ u= \big(I_\alpha \ast F(u) \big)f(u),  \qquad u\in H^1(\mathbb{R}^N),
\end{equation}
which has been widely studied in recent years for its diverse applications in physics.
Here the dimension  $N\geq 2$ and $F\in C^1(\Rset,\Rset)$ is the primitive function of $f$.  The Riesz potential $I_\alpha$ with order $\alpha\in (0,N)$ is defined for
$x \in \mathbb{R}^N \setminus \{0\}$ by
\[
  I_{\alpha}(x)=\frac{A_\alpha}{\abs{x}^{N-\alpha}} \quad \text{ with } \quad
  A_\alpha=\frac{\Gamma(\frac{N-\alpha}{2})}{2^\alpha\pi
  ^{\frac{N}{2}}\Gamma(\frac{\alpha}{2})}
\]
where $\Gamma$ denotes the classical Gamma function and $\ast$ stands for the convolution on \(\Rset^N\).

The Choquard equation \eqref{eqChoquard} is a canonically nonlocal equation due to the appearance of the convolution. When $N=3$, $\alpha=2$ and $F(u)=\abs{u}^2$, as an important model, it arises in various physics contexts, such as the theory of one-component plasma  in describing the electron trapped in its own hole \cite{L} and the model of a polaron at rest in the quantum physics  \cite{Pekar}. This equation also has applications in Einstein--Klein--Gordon system and Einstein--Dirac system
\cite{GG2012}. It is usually referred to as the Schr\"odinger--Newton equation in the description of self-gravitating matter \cite{MPT}. We refer to \cite{Diosi1984,J2,MVSReview} and the references therein for more physical interpretations.

The $p$-homogeneous Choquard equation, e.g., $F(u)=\abs{u}^p$, has already been extensively investigated by a plenty of literatures during last decades. We refer to \cite{L,Lions1980,Lions1984} for the existence and multiplicity results for the Choquard equation \eqref{eqChoquard} in $\mathbb{R}^3$ with $F(u)=\abs{u}^2$  and $\alpha=2$. Moroz and Van Schaftingen \cite{MVJFA} obtained the ground state solution for the $p$-homogeneous Choquard equation. Here and in the sequel, the ground state solution is understood in the sense that it minimizes the energy functional among all nontrivial solutions.
It is shown that nontrivial solutions with finite energies to the $p$-homogeneous Choquard equation can exist if and only if the exponent $p$ satisfies  $\frac{N-2}{N+\alpha}<\frac{1}{p}<\frac{N}{N+\alpha}$, see e.g., \cite{MVJFA,MVTAMS,VanSchaftingenXia}. Several papers have also studied the existence of sign-changing
solutions to the $p$-homogeneous Choquard equation in recent years. In a compact setting, we mention \cite{VanSchaftingenXia} for the least energy sign-changing solution and  refer to \cite{GuiGuo,ZhihuaHangmultiple2017} for the multiple radially symmetric solutions with prescribed number of nodal domains. We refer readers to the survey paper \cite{MVSReview} and the references therein for more recent results in this field.

For the Choquard equation \eqref{eqChoquard}, Moroz and Van Schaftingen \cite{MVTAMS} investigated the ground state solution under the so-called Berestycki--Lions conditions as $N\geq 3$,
\begin{itemize}
  \item [($F_0$)]  $F\in C^1(\Rset,\Rset)$ and there exists $s_0\in \mathbb R $ such that $F(s_0)\neq 0$.
  \item [($f_1$)]  There exists $C>0$ such that for every $s\in \Rset$,  $\abs{sF'(s)}\leq C(\abs{s}^{\frac{N+\alpha}{N}}+\abs{s}^{\frac{N+\alpha}{N-2}}).$
  \item [($f_2$)]   $\lim\limits_{s\to 0}\frac{F(s)}{\abs{s}^{\frac{N+\alpha}{N}}}=0$  and $\lim\limits_{s\to\infty}\frac{F(s)}{\abs{s}^{\frac{N+\alpha}{N-2}}}=0$.
\end{itemize}
Besides the nontrivial assumption $(F_0)$, Battaglia and Van Schaftingen \cite{BV2017} obtained the ground state solution for the planar case by assuming the following hypotheses.
\begin{itemize}
  \item [($f'_1$)]  For every $\theta>0$, there exists $C_\theta>0$ such that for every $s\in\Rset$,
       $$\abs{F'(s)}\leq C_\theta\min\{1,\abs{s}^{\frac{\alpha}{2}}\}e^{\theta \abs{s}^2}.$$
  \item [($f'_2$)]   $\lim\limits_{s\to 0}\frac{F(s)}{\abs{s}^{\frac{2+\alpha}{2}}}=0$.
\end{itemize}
We emphasis that the above assumptions on the nonlinearity are quit mild and reasonable. It turns out that these assumptions are actually "almost necessary" for the existence of nontrivial finite energy solutions \cite{BV2017,MVTAMS}.

On the other hand, some progresses have been made on the saddle type nodal solutions. It was Ghimenti and Van Schaftingen who constructed a surprisingly odd solution for the $p$-homogeneous Choquard equation in their seminal work \cite{GhimentiVanSchaftingen2016}, this kind of odd solution has exactly two half-space nodal domains and makes the issue of constructing sign-changing solutions become more interesting because the compact assumption there is no longer needed.
Actually they took advantage of the reflection group with order $2$ to grantee the nodal property.
Motivated by the highlight results in \cite{GhimentiVanSchaftingen2016}, Wang and the author of current paper recently developed an effective variational approach in constructing saddle type nodal solutions for the $p$-homogeneous Choquard equation in \cite{XiaWang2019,XiaWang2018} by making use of the dihedral group and the full polyhedral groups of the Platonic solids, respectively.  As a result, finding saddle solutions is reduced to solving a minimization problem constrained on the saddle type Nehari manifold.
In a very recent joint work with Zhang \cite{XiaZhang2020}, we developed a unified variational framework for the construction of saddle solutions to the upper critical Choquard equation.
It turns out that the Coxeter symmetry was introduced and these groups contain more reflections than before. All of the symmetric groups used in \cite{GhimentiVanSchaftingen2016,XiaWang2019,XiaWang2018,XiaZhang2020} consist of distinct number of reflections giving different tiling and configurations in $\Rset^N$.
Note that finite Coxeter groups can be generated by finite reflections in $\Rset^N$. Therefore, nodal sets of these saddle solutions are all made up by hyperplanes \cite{GhimentiVanSchaftingen2016,XiaWang2019,XiaWang2018,XiaZhang2020}, each of saddle solutions keeps a constant sign on every conic (sectorial) region and receives opposite signs on any two adjacent regions.
This reveals a distinguished nonlocal feature of the Choquard equation. As a contrast, it is impossible to admit such type of nodal solutions either for its local counterpart of the nonlinear Schr\"odinger equation with the form $-\Delta u +u = |u|^{q-2} u$  \cite{Esteban1982} or for the Yamabe problem $-\Delta u=\abs{u}^{\frac{4}{N-2}} u$ on $\Rset^N$ \cite{Struwe1996}.

Whereas all of the saddle solutions constructed in \cite{GhimentiVanSchaftingen2016,XiaWang2019,XiaWang2018,XiaZhang2020} up until now were only available for the homogeneous Choquard equation. It seems that little is known on the existence of nodal solutions to the Choquard equation \eqref{eqChoquard} except a very recent paper \cite{Zhang2020} in which the compact setting is assumed and a sequence of radial or non-radial sign-changing solutions are obtained.
The condition of Ambrosetti--Rabinowtiz type  is also required in \cite{Zhang2020}  on the nonlinearity.

In the present paper, we shall deal with the existence of saddle solutions for the Choquard equation \eqref{eqChoquard} with a general nonlinearity by introducing the Coxeter symmetry.
The approach for seeking saddle solutions in the previous work heavily depends upon the Nehari manifold method, in turn, the homogeneous or the monotonic property of the nonlinearity plays a significant role. As a consequence, the Nehari manifold method can not work well in tacking with the general Choquard equation \eqref{eqChoquard} if we do not have any additional monotonic assumptions on $F$. This has been observed in the process of finding ground state solutions in \cite{BV2017,MVTAMS}.

To state our results, we shall first recall the Coxeter group and some notations.
 Let $G$ be a finite Coxeter group with rank $k\leq N$ and $\abs{G}$ denote its cardinality.
It follows that $G=\langle \mathcal{S}\rangle$ with $\mathcal{S}$ being the generating set and having $k$ distinct reflections in $\cS$. More precisely, there exists a unique epimorphism $\psi: G\to\{\pm 1\}$ induced by $\psi(s)=-1$ for all $s\in\mathcal{S}$.  A function $u:\Rset^N\to \Rset$ is said to be $G$-symmetry if it satisfies
$$g\diamond u (x)=\psi(g)u(x), \quad \text{ for } g\in G,\quad x\in \Rset^N. $$
Here $\diamond$ denotes  the group action which will be explained in Section \ref{frawork}. With this kind of symmetric groups,  the symmetric critical principle is applicable and we shall consider our problem in a closed $G$-subspace defined by
$$\cH_G=\{u\in H^1(\Rset^N)\st g\diamond u=\psi (g)u, \forall g\in G\}.$$
It turns out that all of the symmetric groups used in \cite{GhimentiVanSchaftingen2016,XiaWang2019,XiaWang2018} are of Coxeter with their ranks less than $4$ yielding different symmetric nodal structures.  In what follows, we say $u\in \cH_G\setminus\{0\}$ is a $G$-saddle solution if $u$ solves the Choquard equation \eqref{eqChoquard} and $u$ is called $G$-groundstate if in addition $u$ minimizes the energy functional amongst all the $G$-saddle solutions.

The first result is about the existence of  $G$-groundstate. Besides the nontrivial assumption $(F_0)$, in the present paper, we assume a little stronger assumptions  than that in \cite{BV2017,MVTAMS} on the nonlinearity $F$, the reason will be explained later.
\begin{itemize}
  \item [($F_1$)]  $\lim\limits_{s\to 0}\frac{f(s)}{s}=\mu \in[0,+\infty)$.
  \item [($F_2$)]  $\lim\limits_{s\to\infty}\frac{F(s)}{\abs{s}^{\frac{N+\alpha}{N-2}}}=0$ for $N\geq 3$;  $\lim\limits_{s\to\infty}\frac{F(s)}{\abs{s}^{\bar{p}}}=0$ for some $\bar{p}\in[2,+\infty)$ as $N=2$.
\end{itemize}
 Our results on the  saddle solutions can be stated as follows.
\begin{Theorem}\label{thm1.1}
Let $N\geq 2$ and $G$ be a finite Coxeter group of rank $k$ with $1\leq k\leq N$. If $F$ is even and satisfies the assumptions $(F_0)$, $(F_1)$ and $(F_2)$, then the Choquard equation
 \eqref{eqChoquard} admits a $G$-saddle solution $u\in H^{1}(\Rset^N)\cap C^2(\Rset^N)$ which turns to be a $G$-groundstate.
\end{Theorem}

We remark that in the $p$-homogeneous case that $F(u)=\abs{u}^p$, Theorem \ref{thm1.1} holds provided that $N\geq 2$ and $p\in[2,\frac{N+\alpha}{N-2})$, which are exactly the same with that in \cite{GhimentiVanSchaftingen2016,XiaWang2019,XiaWang2018}. In fact, the odd solution constructed in \cite{GhimentiVanSchaftingen2016} makes use of the 1-rank Coxeter groups
${\bf{A}_1}$, generated by one reflection. The 2-rank Coxeter group, dihedral group ${\bf{I}}_2(m)$, was used in  \cite[Theorem 1.1]{XiaWang2018} to construct the dihedral type saddle solution that has exactly $2m$ nodal domains. The irreducible Coxeter groups with rank 3  including ${\bf{A}}_3$, ${\bf{B}}_3$, ${\bf{H}}_3$ in \cite{XiaWang2019} yield the saddle solutions with polyhedral symmetries, where ${\bf{A}}_3$, ${\bf{B}}_3$, ${\bf{H}}_3$ are exactly the three kinds of full polyhedral groups for the Plato's polyhedrons in $\mathbb R^3$.
 Different saddle solutions for $p$-homogeneous Choquard equation were also obtained separately by taking advantage of the symmetric groups ${\bf{I}}_2(m)\times {\rm\bf{A}}_1$, ${\bf{I}}_2(m_1)\times {\bf{I}}_2(m_2)$ and ${\rm\bf{A}}_1\times\cdots\times {\rm\bf{A}}_1$  in the previous work \cite{XiaWang2019,XiaWang2018}. In summary, all the symmetric groups used in \cite{GhimentiVanSchaftingen2016,XiaWang2019,XiaWang2018} are types of Coxeter so that each of the saddle solution demonstrates a unique Coxeter's symmetric configurations in $\Rset^N$. The notations and the classification of Coxeter groups will be explained in the final appendix.
 In this sense, Theorem \ref{thm1.1} will further complete the variational framework initiated in \cite{GhimentiVanSchaftingen2016} and further developed in \cite{XiaWang2019,XiaWang2018,XiaZhang2020}. Here we also would like to mention the existence results on sign-changing solutions with higher energies under higher dimensional symmetries \cite{Cingolani2012ZAMP,Clapp2013}. Nevertheless, the key assumption $(Z_0)$ imposed on the group seems necessary and the group action they used can not be realized in $\Rset^3$. Our results manifest that the compactness can still fail below the threshold value in \cite{Cingolani2012ZAMP,Clapp2013} which equals the ground state energy multiplying by the order of the group $G$, see Proposition \ref{energystrictinequality}.

As we shall see  that for the existence of saddle solutions, the higher rank of the Coxeter group we use, the more existing saddle solutions we should consider when restoring the compactness. In fact, as is clarified in \cite{XiaWang2019,XiaWang2018}, given a finite Coxeter group $G$, there exists a threshold value $c_G^*$ below which the compactness can be restored, and the threshold $c_G^*$ actually depends upon the energy levels of the preceding saddle solutions that correspond to subgroups of $G$.  As a result, the existing saddle solutions come into play like a hierarchy when we build one type of saddle solutions with higher dimensional symmetries.

In the process of constructing saddle solutions, our ultimate task is to restore the compactness in $\Rset^N$ under the group action. To this end, a suitable local compactness condition is needed which leads to a prior estimate on the mountain pass type energy level and in turn it attributes to the decay behavior of the sign-changing solutions. For the $p$-homogeneous Choquard equation, as indicated in \cite{GhimentiVanSchaftingen2016}, finding the odd solution relies on the decay behavior of the ground state solution. For the sign-changing solutions of $p$-homogeneous Choquard equation, the decay behavior is only clear when $p\geq 2$, see \cite[Proposition 2.1]{XiaWang2018}.
For the Choquard equation \eqref{eqChoquard}, up until now, only the exponential decay of the ground state solution was known \cite[Proposition 2.2]{YangZhangZhang2017}, which requires the nonlinearity $f$ having superlinear and subcritical growth.  However, there has no decay results for sign-changing solutions of \eqref{eqChoquard}.

In this paper, under our assumptions on the nonlinearity, we shall prove that any weak solutions of \eqref{eqChoquard} has an exponential decay, which plays a big part in determining the threshold for the compactness, see Proposition \ref{energystrictinequality} below.  We remark that this decay estimate particularly holds for sign-changing solutions. This partly complements the results of \cite[Proposition 2.2]{YangZhangZhang2017} where the decay behavior is only considered for positive solutions.

\begin{Theorem}\label{thm1.2}
Let $N\geq 2$, $\alpha\in(0,N)$. If $F\in C^1(\Rset,\Rset)$ satisfies $(F_1)$--$(F_2)$ and $u\in H^1(\Rset^N)$ is a weak solution of \eqref{eqChoquard}. Then, $u\in W^{2,s}_{\rm{loc}}(\Rset^N)$ for any $s\geq 1$ and $u\in C^2(\Rset^N)\cap L^{\infty}(\Rset^N)$.  Moreover, there exist positive constants $C,c>0$ such that
$
 \abs{u(x)}\leq Ce^{-c\abs{x}}.
$
\end{Theorem}
The regularity results have been obtained in \cite{MVTAMS} and the decay estimate depends on a $L^\infty$ estimates on the convolution term and the comparison principle. In particular, as is revealed  in \cite{XiaWang2018}, the additional limit behavior of the convolution term enable us to deal with the quadratic case, which also extends the decay result in \cite{YangZhangZhang2017} for the positive solutions.  However, the decay behavior is still open for the sign-changing solutions to the sub-quadratic Choquard equation \eqref{eqChoquard} and only the groundstate solution to the $p$-homogeneous Choquard equation is proved to have a  polynomial decay \cite{GhimentiVanSchaftingen2016}. This is the reason why we limit the nonlinearity $f$ admitting linear or superlinear growth near the origin in all previous work \cite{XiaWang2019,XiaWang2018}. This also in contrast with the critical Choquard equation \cite{XiaZhang2020}, because it has been shown that all the solutions have the polynomial decay behavior \cite{GHPS2019}.
We also bring attention that the assumptions $(F_2)$ is somewhat stronger than the hypothesis $(f'_1)$ as $N=2$, where only a weaker Moser--Trudinger type growth assumption is assumed. This is because the asymptotic decay behavior of the convolution term $I_\alpha\ast F(u)$ is still unclear for sign-changing solutions.

In a similar manner as in that \cite[Proposition 2.6]{GhimentiVanSchaftingen2016}, we can also prove the constant signed property of the $G$-groundstate on the strict fundamental domain of the Coxeter group $G$ under some stronger assumptions.

\begin{Theorem}\label{thm1.3}
Let $N\geq 2$ and $G$ be a finite Coxeter group of rank $k$ with $1\leq k\leq N$. If $F\in C^1(\Rset,\Rset)$ satisfies $(F_1)$ and $(F_2)$, in addition, if $F'$ is odd and has a constant sign on $(0,+\infty)$, then any $G$-groundstate of the Choquard equation
 \eqref{eqChoquard} maintains a constant sign on the strict fundamental domain and has exactly $\abs{G}$ nodal domains.
 \end{Theorem}

When compared with the local counterpart of the nonlinear Schr\"odinger equations \cite{Weth2006}, one significant difference is that the energy of the ${\bf{A}_1}$-groundstate is strictly less than twice the groundstate energy, i.e., the energy doubling property does not hold any more for the nonlocal Choquard equation. This fact makes it possible to find the global minimal nodal solution which minimizes the functional amongst all the sign-changing solutions. This have been done in \cite{GhimentiVanSchaftingen2016} for the $p$-homogeneous
Choquard equation. It turns out that the minimal nodal solution exist once the Choquard equation \eqref{eqChoquard} has the Nehari structure and the ${\bf{A}_1}$-groundstate.
 To make it clear, we assume that
 \begin{itemize}
  \item[($F'_1$)] $\lim\limits_{s\to 0}\frac{f(s)}{s^{r-1}}=\mu\in[0,+\infty)$ for some $r>2$ and the function $\frac{f(s)}{\abs{s}}$ is increasing in $s\in \Rset\setminus\{0\}$.
  \item [($F'_2$)] there exists some $p\in(2,\frac{N+\alpha}{N-2})$  such that
   $\lim\limits_{\abs{s}\to+\infty}\frac{F(s)}{\abs{s}^p}=0$.
\end{itemize}

 \begin{Theorem}\label{thm1.4}
Let $N\geq 2$ and $\alpha\in (0,N)$. If $F\in C^1(\Rset,\Rset)$ is even and satisfies $(F'_1)$ and $(F'_2)$, then the Choquard equation \eqref{eqChoquard}
has a global minimal nodal solution.
\end{Theorem}

We sketch our main idea and organize the paper as follows. By introducing the Coxeter group $G$ with rank $k\leq N$, we shall furnish a suitable framework space for searching the $G$-saddle solutions, see Section \ref{frawork}. We then prove Theorem \ref{thm1.2} in Section \ref{exponent} that all the solutions of the Choquard equation \eqref{eqChoquard} decay exponentially at infinity. In addition, the regularities for the solutions enable us to establish the Pohozaev identity for this equation.  As is mentioned above, the lack of monotonic property on the nonlinearity makes the Nehari manifold method can not work well any more. We shall consider the mountain pass level as in \cite{BV2017,MVTAMS} rather than the minimization on the corresponding Nehari manifold, this is possible since the nontrivial assumption $(F_0)$ yields the mountain pass geometry, see the details in Section \ref{frawork}.
To avoid the condition of Ambrosetti--Rabinowtiz type on the nonlinearity, we shall construct the Pohozaev--Palais--Smale sequence at the mountain pass level $c_G$ by using a scaling technique, see Proposition \ref{LpropPSsequence}. To prove the key energy estimate on this level, we shall construct test functions by using the preceding saddle solutions, which can be done by an induction argument, see Section \ref{secenerest}. More precisely, we first show that $c_{{\bf{A}}_1}<2c_0$ by using the exponential decay of the ground state solution for the Choquard equation \eqref{eqChoquard}. And then, we conclude the existence of a nontrivial saddle solution with ${\bf{A}_1}$-symmetry by employing the concentration compactness principle. As in \cite{MVTAMS}, a crucial minimax characterization on this mountain pass level $c_G$ will be given, which reveals that $c_{\bf{A}_1}$ actually equals the infimum of the energy functional constrained on the corresponding Pohozaev manifold with ${\bf{A}_1}$-symmetry, this also involves the scaling technique, see Proposition \ref{energycharacterization} below. Therefore,
$c_{\bf{A}_1}$ is indeed the least energy amongst  ${\bf{A}}_1$-saddle solutions. The existence of  ${\bf{A}}_1$-groundstate will be the start point to construct the saddle solution with the dihedral nodal structures. By repeating this procedure, we then conclude step by step the existence of $G$-saddle solutions with the rank $k\leq N$. Every saddle solutions  that we obtained turns out to be the $G$-groundstate, see Proposition \ref{energycharacterization}. Theorem \ref{thm1.1} will be inductively completed in Section \ref{section5} by employing the concentration compactness principle.  We also prove Theorem \ref{thm1.3} in Section \ref{section5} under the additional assumptions on the nonlinearity.
We finally complete the proof of Theorem \ref{thm1.4}. This will be done by considering the minimization problem on the set of nodal solutions,
$$c_{\rm{nod}}=\inf_{\cM_{\rm{nod}}}\cE$$
where the constraint $\cM_{\rm{nod}}=\{u\in H^1(\Rset^N)\st u^\pm\neq 0, \cE'(u)=0\}$ with $u^\pm=\max\{\pm u, 0\}$.
Up to translations, this problem can be achieved along any minimizing sequence. We shall see that this is a direct consequence of the fact that $c_{\rm{nod}}\leq c_{{\bf{A}}_1}<2c_0$.
Since the background knowledge on the theory of Coxeter group has been summarized in \cite{XiaZhang2020}, we give an appendix at the end of this paper for the convenience of readers.

\section{Variational framework and minimax characterizations}
\label{frawork}

The Choquard equation \eqref{eqChoquard} has a  variational structure and corresponds to an energy functional $\mathcal{E}$, defined for each function $u\in H^1(\Rset^N)$ by
$$\mathcal{E}(u)=\frac{1}{2}\int_{\Rset^N}\abs{\nabla u}^2+\abs{u}^2-\frac{1}{2}\int_{\Rset^N}(I_\alpha* F(u))F(u).$$
According to the arguments in \cite{BV2017,MVTAMS}, this functional is well defined and is of $C^1$ in $H^1(\Rset^N)$.
In the sprite of Berestycki--Lions assumptions, the authors in \cite{BV2017,MVTAMS} investigated the existence of ground state solutions for the Choquard equation \eqref{eqChoquard} by employing a scaling technique introduced by Jeanjean \cite{Jeanjean1999} to construct the so called Pohozaev--Palais--Smale sequence. Precisely, they proved that
the mountain pass level defined below actually yields the ground state solution, it takes the form of
\begin{equation*}\label{minimizationground}
c_{0}=\inf_{\gamma\in \Gamma_0 }\max_{t\in[0,1]}\mathcal{E}(\gamma(t))
\end{equation*}
where the paths set is defined by
$$
\Gamma_0=\{\gamma\in C([0,1];H^1(\Rset^N)) \st \gamma(0)= 0 \text{ and } \mathcal{E}(\gamma(1))<0\}.
$$

To obtain the symmetric nodal structures, we shall introduce the Coxeter symmetry, see the details in the Appendix.
For $k\geq 1$, we denote $\mathcal{G}_k$ and $\mathcal{W}_k$  be the collections of finite Coxeter groups and the irreducible Coxeter groups with rank $k$, respectively. Given $G\in \mathcal{G}_k$,
by Proposition \ref{proparabolic}, we see that $G=W_{k_1}\times\cdots\times W_{k_\ell}$ where $W_{k_i}\in\mathcal{W}_{k_i}$ for each $k_i\geq 1$ and $\sum_i k_i=k$.
Recalling the unique epimorphism $\psi:G\to\{\pm 1\}$ defined in Lemma \ref{lemmaepimorphism}, we now set up the framework space
$$
\mathcal{H}_{G}:=\{u\in H^1(\Rset^N)\st \texttt{g} \diamond u=\psi(g) u \text{ for all } \texttt{g}\in G\}.
$$
Here the group action on a function $u:\Rset^N\to \mathbb R$ is defined for each $g\in G$ by
$$
g\diamond u(x)=u(g^{-1}x)
\quad\text{
 with
}\quad
gx:=(g\oplus1_{N-k})x.
$$

The energy functional permits the mountain pass geometry, so that the mountain pass level is a natural candidate for seeking weak solutions. Namely, we consider
\begin{equation}\label{minimizationprob}
c_{G}=\inf_{\gamma\in \Gamma_G }\max_{t\in[0,1]}\mathcal{E}(\gamma(t))
\end{equation}
where  paths set belongs to $\cH_G$, i.e., $\Gamma_G=\{\gamma\in \Gamma_0 \st \gamma(t)\in \cH_G \text{ for all } t\in [0,1]\}$.
Comparing with the saddle type nodal solutions obtained in \cite{MVJFA,XiaWang2018,XiaWang2019}, we have neither monotonic assumption nor Ambrosetti--Rabinowitz type growth condition on the nonlinearity. To overcome this, we recall the Pohozaev identity, which is allowed by the regularities exhibited in Theorem \ref{thm1.2}. For the detailed proofs, we refer to \cite[Theorem 3]{MVTAMS} and \cite[Proposition 5.2]{BV2017}.
\begin{proposition}\label{proppohozaevidentity}
Let $N\geq 2$. If $F$ satisfies $(F_1)$--$(F_2)$ and $u\in H^1(\mathbb R^N)\cap W^{2,2}_{\rm{loc}}(\mathbb R^N)$ solves the equation \eqref{eqChoquard}, then
\begin{equation}\label{pohozaev}
\mathcal{P}(u)=\frac{N-2}{2}\int_{\Rset^N}\abs{\nabla u}^2+\frac{N}{2}\int_{\Rset^N}\abs{u}^2 -\frac{N+\alpha}{2}\int_{\Rset^N}(I_\alpha*F(u))F(u)=0.
\end{equation}
\end{proposition}
We define the Pohozaev manifold that
$$\cP=\{u\in H^1(\Rset^N )\setminus\{0\}\st \cP(u)=0\}$$
and set $\cP_G=\cP\cap \mathcal{H}_G$.
To simplify our notations, we define
$$ A(u)=\int_{\Rset^N}\abs{\nabla u}^2, \quad B(u)=\int_{\Rset^N}\abs{u}^2, \quad  \text{ and }\quad  Q(u)=\int_{\Rset^N}(I_\alpha*F(u))F(u).$$

\begin{lemma}\label{pathofpohozaev}
Let $N\geq 2$ and $F\in C^1(\Rset;\Rset)$ satisfy the assumptions $(F_1)$--$(F_2)$. Then for each $u\in \cH_G\setminus\{0\}$ that satisfies $Q(u)>0$, there exist a path $\gamma_u\in C([0,+\infty);\cH_G)$ and a unique $t_u\in (0,+\infty)$ such that
$\gamma_u(0)=0$, $\gamma_{u}(t_u)\in \cP_G$  and $\cE(\gamma_{u}(t_u))=\max_{t\geq 0}\cE(\gamma_u(t))$. Particularly, if $\cP(u)\leq 0$, then $t_u\leq 1$.
\end{lemma}
This lemma essentially is the results of \cite[Propositons 4.1]{MVTAMS} (see \cite[Propositons 5.1]{BV2017} for $N=2$) and the proof relies on the scaling technique developed by Jeanjean and Tanaka \cite{Jeanjean2003}.
\begin{proof}
Let $u\in \cH_G\setminus\{0\}$ that satisfies $Q(u)>0$. We define a path
$\gamma_u :[0,+\infty)\to \cH_G$ such that for $N\geq 3$
$$\gamma_u(t)(x)=
\begin{cases}
u(x/t),& \text{if } t>0,\\
0,& \text{if } t=0.
\end{cases}
$$
Then the path $\gamma_u(t)\in C([0,+\infty);\cH_G)$ and satisfies $\gamma_u(0)=0$. Direct computations show us two functions that
$$\alpha(t):=\cE(\gamma_u(t))=\frac{t^{N-2}}{2}A(u)+\frac{t^N}{2}B(u)-\frac{t^{N+\alpha}}{2}Q(u)$$
and $$\beta(t):=\cP(\gamma_u(t))=\frac{N-2}{2}t^{N-2}A(u)+\frac{N}{2}t^NB(u)-\frac{N+\alpha}{2}t^{N+\alpha}Q(u).$$
It is easy to see that $\beta(t)=t\alpha'(t)$ and
$\alpha(t)$ admits a global maxima because $\alpha(t)>0$ for small $t$ and $\alpha(t)<0$ when $t$ is sufficient large. It then follows that there exists $t_u\in(0,+\infty)$ such that $\alpha(t_u)=\max_{t\geq 0}\alpha(t)$. Moreover, $t_u$ is unique since the function $\beta(t)$ only has one zero point so that $\beta(t_u)=0$ and $\gamma_u(t_u)\in \cP_G$. In fact, suppose in contrary that $t_u^1<t_u^2$ such that
$\beta(t_u^1)=\beta(t_u^2)=0$, we then deduce a contradiction that
$$\Big(\frac{1}{(t_u^1)^2}-\frac{1}{(t_u^2)^2}\Big)\frac{N-2}{2}A(u)=\Big((t_u^1)^\alpha-(t_u^2)^\alpha\Big)\frac{N+\alpha}{2}Q(u).$$
Note that $\beta(t)>0$ on the interval $(0,t_u)$. $\cP(u)\leq 0$ implies that $\beta(1)\leq 0$, this brings us that $(0,t_u)\subset(0,1)$, which amounts to $t_u\leq 1$.

For the case $N=2$, the above path $\gamma_u$ is not anymore continuous at $0$ in $\cH_G$, we shall connect the function $\gamma_u(t_u)$ with $0$ by the multiplication. Since $Q(u)>0$, we can find a $t_u>0$ such that $\beta(t_u)=0$, so that $\gamma_u(t_u)\in \cP_G$. We then define that
$$\tilde{\gamma}_u(t)(x)=
\begin{cases}
\frac{t}{\tau_0}u(\frac{x}{\tau_0t_u}),& \text{if } t\leq \tau_0,\\
u(\frac{x}{tt_u}),& \text{if } t\geq \tau_0.
\end{cases}
$$
The parameter $\tau_0\in (0,1)$ will be determined later. Clearly, we see that $\gamma_u$ is continuous in $\cH_G$. For $t\geq \tau_0$, the Pohozaev identity gives that
\begin{align*}
\tilde{\alpha}(t)=\cE(\tilde{\gamma}_u(t))&=\frac{1}{2}A(u)+\frac{(t_ut)^2}{2}B(u)-\frac{(tt_u)^{2+\alpha}}{2}Q(u)&\nonumber\\
&=\frac{1}{2}A(u)+\Big(\frac{t^2}{2}-\frac{t^{2+\alpha}}{2+\alpha}\Big)t_u^2B(u).
\end{align*}
It is not difficult to check that $\tilde{\alpha}(t)$ achieves its strict maximum at $t=1$ and is negative for sufficient large $t$.
 Without loss of generality, we assume that $t_u=1$. For
$t\leq \tau_0$, by the assumptions $(F_1)$--$(F_2)$ and the Hardy--Littlewood--Sobolev inequality,  we deduce that
 \begin{align*}
\frac{\tau_0}{t}\int_{\Rset^2} (I_\alpha*&F(tu/\tau_0)f(tu/\tau_0)u=\frac{\tau^2_0}{t^2}\int_{\Rset^2}(I_\alpha*F(tu/\tau_0)f(tu/\tau_0)tu/\tau_0\nonumber\\
 &\leq C  \frac{\tau^2_0}{t^2} \Big(\int_{\Rset^2}\abs{F(tu/\tau_0)}^{\frac{4}{2+\alpha}}\Big)^{\frac{2+\alpha}{4}}\Big(\int_{\Rset^2}\abs{f(tu/\tau_0)tu/\tau_0}^{\frac{4}{2+\alpha}}\Big)^{\frac{2+\alpha}{4}}\nonumber&\\
&\leq C  \frac{\tau^2_0}{t^2}\Big(\int_{\Rset^2}\abs{tu/\tau_0}^{\frac{8}{2+\alpha}}+\abs{tu/\tau_0} ^{\frac{4\bar{p}}{2+\alpha}}\Big)^{\frac{2+\alpha}{2}}\leq C \big(\|u\|^4+\|u\|^{2\bar{p}}\big).
\end{align*}
We  conclude our proof by fixing a proper $\tau_0<1$ such that $\frac{\dif }{\dif \tau}\cE(\tilde{\gamma}_u(t))>0$ for $t\leq \tau_0$. Actually, we have for sufficiently small $\tau_0$,
 \begin{align*}
 \frac{\dif }{\dif t} \tilde{\alpha}(t)&=\frac{\dif }{\dif t} \Big(\frac{t^2}{2\tau_0^2}A(u)+\frac{t^2}{2}B(u)-\frac{\tau_0^{2+\alpha}}{2}\int_{\Rset^2}(I_\alpha*F(tu/\tau_0))F(tu/\tau_0)\Big)&\nonumber\\
&= t\Big(\frac{1}{\tau_0^2}A(u)+B(u)-\tau_0^\alpha \frac{\tau_0}{t} \int_{\Rset^2}(I_\alpha*F(tu/\tau_0)f(tu/\tau_0)u\Big)\nonumber&\\
&\geq  t\Big(\frac{1}{\tau_0^2}A(u)+B(u)-C\tau_0^\alpha (\|u\|^4+\|u\|^{2p})\Big)>0.
\end{align*}
Hence, $\tilde{\alpha}(t)$ is strictly increasing when $t\leq \tau_0$,  which allows $t=t_u$ being the unique global maximum point.
\end{proof}

A standard argument shows that the functional $\mathcal{E}$ has the mountain pass geometry, which furnishes a Pohozaev--Palais--Smale sequence at the energy level $c_G$ defined by \eqref{minimizationprob}. This sequence has an additional property that the equation \eqref{pohozaev} asymptotically holds.
\begin{proposition}\label{LpropPSsequence}
Let $N\geq 2$, $\alpha\in(0,N)$. Assume the growth assumptions $(F_0)$--$(F_2)$ hold.
Then there exists a sequence $(u_n)_{n\in\Nset}$ in $\mathcal{H}_G$ such that, as $n\to\infty$,
\begin{align*}\label{PPS}
   \mathcal{E}(u_n)\to c_{G}\in(0,+\infty), \; \mathcal{P}(u_n)\to 0, \;\text{\rm{ and} }  \mathcal{E}'(u_n)\to 0 \;\text{\rm{ strongly in the dual space }}  \mathcal{H}_G^{*}.
\end{align*}
\end{proposition}
\begin{proof}

To show $c_G<+\infty$, we need construct a path in $\cH_G$. By Lemma \ref{pathofpohozaev}, this turns out to find a function $w_0\in \cH_G$ such that $Q(w_0)>0$.
 In order to derive this, by recalling the nontrivial assumption $(F_0)$ and the unit ball $B_1\subset\Rset^N$ centered at the origin, we take $w_0=s_0\sum_{g\in G} \psi(g) g\diamond \chi_{B_1\cap \mathcal{F}}(x)$ so that $w_0\in \cH_G$ and $Q(w_0)=F(s_0)^2Q(\chi_{B_1})>0$. Here $\mathcal{F}$ is the strict fundamental domain of the Coxeter group $G$ and $\chi_{B_1\cap\cF}$ denotes the characteristic function on $B_1\cap\cF$. In fact, for any $\bar{g}\in G$, we have
$$\bar{g}\diamond w_0=s_0\sum_{g\in G} \psi(g) (\bar{g}g)\diamond \chi_{B_1\cap \mathcal{F}}(x)=s_0\psi(\bar{g})\sum_{g\in G} \psi(\bar{g}g) (\bar{g}g)\diamond \chi_{B_1\cap \mathcal{F}}(x)=\psi(\bar{g})w_0.$$
The remanent proof can be verbatim copied from \cite[Proposition 2.1]{MVTAMS} and \cite[Proposition 3.1]{BV2017}, respectively. The only difference is the Sobolev space $H^1(\Rset^N
)$ should be replaced by the framework subspace $\cH_G$, so we omit the details but emphasis that the scaling technique of Jeanjean \cite{Jeanjean1999} can be used to get the additional asymptotic property for the Palais--Smale sequence that $\cP(u_n)\to 0$ as $n\to\infty$.
\end{proof}

We remark that the nontrivial assumption $(F_0)$ indicates both the paths set $\Gamma_G$ and the Pohozaev manifold $\cP$ are nonempty.  We shall prove the following characterizations.
\begin{proposition} \label{energycharacterization}
Let $N\geq2$, $\alpha\in(0, N)$. Assume that the assumptions $(F_0)$--$(F_2)$ hold. Then
$$c_G=p_G=m_G.$$
Here $p_G=\inf_{\mathcal{P}_G}\mathcal{E}$ and $m_G=\inf_{\mathcal{M}_G }\mathcal{E}$ with $\mathcal{M}_G=\{ u\in \mathcal{H}_G\setminus\{0\} \st \mathcal{E}'(u)=0 \}.$
\end{proposition}

\begin{proof} For any $u\in \cP_G$, we have $\cP(u)=0$ so that $Q(u)>0$. By Lemma \ref{pathofpohozaev} and a suitable change of variable, we conclude that $\gamma_u(t)\in \Gamma_G$. It then follows that $c_G\leq \max_{t\geq 0}\cE(\gamma_u(t))=\cE(u)$. The arbitrariness of $u\in \cP_G$ implies that  $c_G\leq p_G$. We easily deduce that $p_G\leq m_G$ since $\cM_G \subset \cP_G$.  The last inequality $m_G\leq c_G$ will be proved by the existence of a $G$-saddle solution $u\in \cH_G$ that satisfies $\cE(u)\leq c_G$. It turns out that this proposition will be complete by an induction argument.
\end{proof}

\section{Exponential decays for the solutions}
\label{exponent}
We shall prove that solutions of Choquard equation \eqref{eqChoquard} decay exponentially at infinity,
which shall play a big part in determining the threshold values for the compactness, see Proposition \ref{energystrictinequality} below.

\noindent\textbf{Proof of Theorem \ref{thm1.2}.}
The local regularity that $u\in W^{2,s}_{\rm{loc}}(\Rset^N)$ for $s\geq 1$ follows from \cite[Theorem 2]{MVTAMS} for $N\geq 3$ and from \cite[Proposition 5.1]{BV2017} for $N=2$, respectively.  Next, we show that $I_\alpha*F(u)\in L^\infty(\Rset^N)$. For $N\geq 3$, let $H(u)=F(u)/u$ and $K(u)=f(u)$. In view of the assumptions $(F_1)$--$(F_2)$, we deduce that there exists constant $C>0$ such that for $x\in\Rset^N$,
$$\abs{H(u)},\abs{K(u)}\leq C \big(\abs{u(x)}^{\frac{\alpha}{N}}+\abs{u(x)}^{\frac{2+\alpha}{N-2}}\big).$$
It then follows that $H,K\in L^{\frac{2N}{\alpha}}(\Rset^N)+L^{\frac{2N}{2+\alpha}}(\Rset^N)$. Here a function $g\in L^{p_1}(\Rset^N)+L^{p_2}(\Rset^N)$ means that $g\leq \abs{g_1}+\abs{g_2}$ with $g_i\in L^{p_i}(\Rset^N)$ for $i=1,2$. We then conclude by \cite[Proposition 3.1]{MVTAMS} that $u\in L^p(\Rset^N)$ for any $p\in [2,\frac{N}{\alpha}\frac{2N}{N-2})$. For $N=2$, we conclude directly by the Rellich--Kondrachov embedding theorem that $u\in L^p(\Rset^N)$ for any $p\in [2,\infty)$.

We denote for convenience $\frac{2N}{N-2}=\infty$ in case that $N=2$. By the growth assumptions $(F_1)$--$(F_2)$, there exists a constant $C>0$  such that for all $t\in \Rset$
$$\abs{F(t)}\leq C(\abs{t}^2+\abs{t}^{\frac{N+\alpha}{N-2}}).$$%
For the planar case, the exponent $\frac{N+\alpha}{N-2}$ will be replaced by fixed $\bar{p}\in[2,\infty)$.
We then conclude that  $I_\alpha*F(u)\in L^\infty(\Rset^N)$ by proving that  $I_\alpha*\abs{u}^s\in L^{\infty}(\Rset^N)$ for any $s\in (\frac{2\alpha}{N},\frac{2N}{N-2})$.
In fact, we have by direct computations that
\begin{equation}\label{eqboundedRiesz}
\begin{split}
\int_{\Rset^N} \frac{\abs{u(y)}^s}{\abs{x-y}^{N-\alpha}} \dif y &=\int_{\Rset^N}\frac{\abs{u(x-y)}^s}{\abs{y}^{N-\alpha}}\dif y\\
&=\int_{B_1}\frac{\abs{u(x-y)}^s}{\abs{y}^{N-\alpha}} \dif y +\int_{B_1^c}\frac{\abs{u(x-y)}^s}{\abs{y}^{N-\alpha}} \dif y.
\end{split}
\end{equation}
By choosing $t\in (\frac{N}{\alpha},\frac{N}{s\alpha}\frac{2N}{N-2})$ and $r\in (\frac{2}{s},\frac{N}{\alpha})$ such that $\frac{t}{t-1}<\frac{N}{N-\alpha}<\frac{r}{r-1}$,
we deduce that  $\abs{x}^{\alpha-N}\in L^{\frac{t}{t-1}}(B_1)$ and $\abs{x}^{\alpha-N}\in L^{\frac{r}{r-1}}(B^c_1)$.
Here and in the sequel, the notation $B_r$ denotes the ball in $\Rset^N$ with radius $r$ and centered at the origin.
By the H\"older inequality, we have
\begin{equation}\label{eqboundedRieszc}
\begin{split}
\int_{\Rset^N}\frac{\abs{u(y)}^s}{\abs{x-y}^{N-\alpha}}\dif y&\leq  \Big(\int_{B_1}\abs{x}^{\frac{t(\alpha-N)}{t-1}}\dif x \Big)^{\frac{t-1}{t}}\Big(\int_{\Rset^N}\abs{u}^{st}\Big)^{\frac{1}{t}}\\
&\quad +\Big(\int_{B^c_1}\abs{x}^{\frac{r(\alpha-N)}{r-1}}\dif x \Big)^{\frac{r-1}{r}}\Big(\int_{\Rset^N}\abs{u}^{sr}\Big)^{\frac{1}{r}}.
 \end{split}
 \end{equation}
Hence, \eqref{eqboundedRiesz} and \eqref{eqboundedRieszc} conclude the  claim.
Note that $u\in W^{2,s}_{\rm{loc}}(\Rset^N)$ for $s\geq 1$. We then derive by combining the Rellich--Kondrachov embedding theorem and a standard Schauder estimate that $u\in C^{2,\lambda}(B_1(x))$ for some $\lambda\in(0,1)$ and there exists a constant $C$ independent of $x$ such that
\begin{equation*}
\abs{u(x)}\leq \|u\|_{C^{1,\lambda}(B_1(x))}\leq C\|u\|_{W^{2,s}(B_1(x))}.
\end{equation*}
This leads to $\lim_{\abs{x}\to+\infty}\abs{u(x)}=0$. We thus have $u\in L^\infty(\Rset^N)$ since $u\in C^1(\Rset^N)$.
Moreover, for $s\in [2, \frac{N+\alpha}{N-2}]$, we can take $\varepsilon\in (0,s-\frac{2\alpha}{N})$ so that
\begin{equation*}\label{eqasymptoticRiesz}
\begin{split}
\int_{\Rset^N}&\frac{\abs{u(y)}^s}{\abs{x-y}^{N-\alpha}}\dif y
= \int_{B_{\abs{x}/2}}\frac{\abs{u(y)}^s}{\abs{x-y}^{N-\alpha}}\dif y
+\int_{B^c_{\abs{x}/2}}\frac{\abs{u(y)}^s}{\abs{x-y}^{N-\alpha}}\dif y\\
&\leq \C\abs{x}^{\alpha-N}\int_{\Rset^N}\abs{u}^s + \|u\|^{\varepsilon}_{L^\infty( B^c_{\abs{x}/2})}
(I_\alpha*\abs{u}^{s-\varepsilon})(x).
\end{split}
\end{equation*}
This gives that  $\lim_{\abs{x}\to +\infty} (I_\alpha*F(u))(x)=0$, which, together with the growth assumption $(F_1)$, amounts to
$$
\lim_{\abs{x}\to+\infty}\big(I_\alpha*F(u)\big)\abs{f(u)/u} =0.
$$
Therefore, there exists  $\gamma\in (0,1)$ such that
$(I_\alpha*F(u))\abs{f(u)}\leq \gamma \abs{u(x)}$ for sufficiently large $x$.
Then the desired decay estimate can be derived by using the Kato's inequality \cite{Kato1972} and by virtue of the comparison principle, we omit the details and refer to \cite{MVJFA,XiaWang2018}.

\section{Energy estimates}
\label{secenerest}

\begin{proposition}\label{energystrictinequality}
Let $N\geq 2$ and $\alpha\in(0,N)$. Assume that $G\in \mathcal{G}_k$ with $1\leq k\leq N$.
Then there exists $c_G^*$ such that
\begin{align*}
0< c_{G}<c_G^*&:=\min\big\{\abs{Gq}c_{S_{q}} \big| q\in \partial^{k-1}\mathcal{F}\big\}<\min\big\{\abs{Gq}c_{S_{q}} \big| q\in \partial^{i}\mathcal{F},i=0,1,\cdots,k-2\big\}.
\end{align*}
\end{proposition}
According to the classification results of the finite Coxeter group, $\mathcal{W}_1$ contains only ${\rm\bf{A}}_1$ and its fundamental domain is $\Rset^N_+$. In this case, the threshold  value $c_{G}^*=2c_0$ \cite[Proposition 2.4]{GhimentiVanSchaftingen2016}. For $k=2$, $\mathcal{W}_2$ is consist of dihedral groups ${\rm\bf{I}}_2(m)$ for $m\geq 2$ and the associated fundamental domain is a sectorial region with dihedral angle ${\pi}/{m}$. The last two inequalities in this case turn out to be $c_G<c_G^*:=mc_{{\bf{A}_1}}<2mc_0$ \cite[Proposition 3.1]{XiaWang2018}.

\begin{proof}
We shall construct test functions for each rank by induction. We remark that Proposition \ref{energycharacterization} holds for $G=\{1\}$, see \cite{BV2017,MVTAMS}. This will be the start point. We first take a cut-off function $\zeta\in C_c^2(\Rset^N;[0,1])$ such that $\zeta=1$ on $B_1$ and $\supp\zeta\subset B_{2}$.
We begin our construction from $k=1$. We then have $G=\rm\bf{A}_1$ and $\mathcal{F}=\Rset^N_+$.
Let $U_0$ be the ground state solution for the Choquard equation \eqref{eqChoquard}.
We then define for $x=(x',x^N)\in \Rset^N$ that
$$
U_{R}(x):= \big(\zeta_RU_{0}\big)(x',x^N-3R)- \big(\zeta_RU_{0}\big)(x',-x^N-3R).
$$
It is easy to check that $\tau_N\diamond U_R=-U_R$ where $\tau_N$ is the reflection such that $\tau_N(x',x^N)=(x',-x^N)$.
Repeating the subsequent arguments, we conclude that $c_{{\rm\bf{A}}_1}<2c_0$. With this estimate, we shall prove successively that the Choquard equation \eqref{eqChoquard} admits a saddle solution $U_{{\rm\bf{A}_1}}$ such that $\mathcal{E}(U_{{\rm\bf{A}_1}})\leq c_{{\rm\bf{A}_1}}$, which implies that Proposition \ref{energycharacterization} holds for $k=1$.

Recall that $G=W_{k_1}\times\cdots\times W_{k_\ell}:=W_{k_1}\times W$. We now suppose this proposition holds for $k-1$.
By Lemma \ref{lemmainduction}, for each $q\in \partial^{k-1}\mathcal{F}\cap \partial B_1$, we see that $S_q\in \mathcal{G}_{k-1}$.
Without loss of generality, we may assume $\partial^{k-1} \mathcal{F}\subset \cap_{s\in \mathcal{S}_{k-1}\subset \mathcal{S}} H_s$ with  $\mathcal{S}\setminus\mathcal{S}_{k-1}=\{s_1\}$ and $s_1\in W_{k_1}$.
We carry out the proof into two subcases. The first case is $k_1=1$.  We thus have $S_q=W$ and $q$ is perpendicular to $ H_{s_1}$. Hence, we may assume that $H_{s_1}=\{(x_1,x')\st x_1=0\}$. By the inductive assumption, the Choquard equation \eqref{eqChoquard} admits saddle solutions $U_{S_{q}}\in H^{1}(\Rset^{N})$ that achieves $c_{S_{q}}$ and satisfies that
$$
w \diamond U_{S_{q}}= \psi(w) U_{S_{q}}, \qquad \forall w\in W,
$$
where $\psi: G \to \{\pm 1\}$ is the epimorphism. Particularly, Proposition \ref{energycharacterization} holds for any proper subgroup  of $G$, that is, $c_{S_{q}}=p_{S_q}=m_{S_q}$.
We define
$$
U^{q}_{R}(x):= \big(\zeta_RU_{S_{q}}\big)(x^1-3R,x')- \big(\zeta_RU_{S_{q}}\big)(-x^1-3R,x').$$
We then conclude $s_1\diamond U^q_R=-U_R^q$ and $w\diamond U_R^q=\psi(w)U^q_R$ for any $w\in W$, so that $U^q_R\in\mathcal{H}_G$ since
$s_1w=ws_1$ for $w\in W$. Here the group action $S_q$ is realized in the space $\{0\}\times\Rset^{k-1}=\Rset^{k-1}$.

We now deal with the remaining case that $k_1\geq 2$. We realize the group action $S_q$ in the whole space $\Rset^k$. In other words, we assume the saddle solution $U_{S_q}\in H^1(\Rset^N)$ that achieves $c_{S_q}$ satisfies
$$
w \diamond U_{S_{q}}= \psi(w) U_{S_{q}}, \qquad \forall\; w\in S_{q}.
$$
Note that $W_{k_1}$ is irreducible and $s_1\not\in S_q$. By the Coxeter's classification theorem \ref{classification} (or by the connectedness of the Coxeter diagrams), we see that there exists at most three subgroups of $W_{k_1}$ such that $S_q=W_{t_1}\times W_{t_2}\times W_{t_3}\times W$, here $W_{t_i}\in \mathcal{W}_{t_i}$, $t_1\geq 1$ and $t_2,t_3\geq 0$ satisfy that  $t_1+t_2+t_3=k_1-1$.
Recalling Lemma \ref{normalsubgroup}, we then define a test function
$$
U^{q}_{R}(x):=\frac{1}{\abs{S_{q}}}\sum_{n\in N_{k}}\big(\zeta_RU_{S_{q}}\big)(n^{-1}s_1x-l_G Rq)-\big(\zeta_RU_{S_{q}}\big)(n^{-1}x-l_G Rq)
$$
where $N_{k}=N_{k_1}\times W$ and the constant $l_G$ is determined so that the balls $B_{2R}(l_GRGq)$ are separated far away from each other.
Here we recall that $Gq$ is the orbit of $q\in \partial^{k-1}\cF\cap \partial B_1$ under the action of $G$. This can be done provided that $l_G$ is sufficiently large, since there exist constants $K_1,K_2>0$ that depend only on the group $G$ such that for $x\neq y\in Gq$,
 $$K_1:=\min_{x\neq y\in Gq}\abs{x-y}\leq \abs{x-y}\leq \max_{x\neq y\in Gq}\abs{x-y}:=K_2.$$
 We claim that $U^{q}_R\in \mathcal{H}_G$.
By the facts that $wW_{k_1}=W_{k_1}w$ and $wq=q$ for $w\in W$, we deduce
\begin{equation*}
w^{-1}\diamond U^{q}_R=\frac{1}{\abs{S_{q}}}\sum_{n\in N_k}\big(\zeta_RU_{S_{q}}\big)(n^{-1}s_1wx-l_G Rq)-\big(\zeta_RU_{S_{q}}\big)(n^{-1}wx-l_G Rq)=\psi(w)U_R^{q}.
\end{equation*}
 On the other hand, for every $s\in \mathcal{S}\cap (W_{t_1}\times W_{t_2}\times W_{t_3})$, it follows that $s\in W_{k_1}=N_{k_1}\rtimes \langle s_1\rangle $, so that
 there exist unique $n_s, n_s'\in N_{k_1}$ such that $s=s_1n_s=n_s's_1$.
 We then have
   \begin{equation*}\begin{split}
     &s\diamond U^{q}_R=\frac{1}{\abs{S_{q}}}\sum_{n\in N_k}\big(\zeta_RU_{S_{q}}\big)(n^{-1}s_1sx-l_G Rq)-\big(\zeta_RU_{S_{q}}\big)(n^{-1}sx-l_G Rq)\\
  &=\frac{1}{\abs{S_{q}}}\sum_{n\in N_k}\big(\zeta_RU_{S_{q}}\big)(n^{-1}n_sx-l_G Rq)-\frac{1}{\abs{S_{q}}}\sum_{n\in N_k}\big(\zeta_RU_{S_{q}}\big)( n^{-1}n_s's_1x-l_G Rq)\\
  &=-\frac{1}{\abs{S_{q}}}\sum_{n\in N_k}\big(\zeta_RU_{S_{q}}\big)(n^{-1}s_1x-l_G Rq)+\frac{1}{\abs{S_{q}}}\sum_{n\in N_k}\big(\zeta_RU_{S_{q}}\big)( n^{-1}x-l_G Rq)=-U_R^{q}.
  \end{split}
\end{equation*}
By the construction, we easily see $s_1\diamond U_{R}^q=-U_{R}^q$.
We thus deduce $s\diamond U_{R}^{q}=\psi(s)U_{R}^{q}$ for any $s\in \mathcal{S}$. This concludes the claim that $U^{q}_R\in \mathcal{H}_G$.

We now prove that $Q(U_R^q)>0$ for sufficiently large $R$. To this end,  we denote $B^j_{2R}=B_{2R}(l_GR{z}_j+x^0)$ for given $x^0\in\{0\}\times\Rset^{N-k}$ and ${z}_j\in Gq$. Note that all of the balls $B^j_{2R}$ for $1\leq j\leq \abs{Gq}$ are separated from each other.
We then deduce that
\begin{equation*}
\begin{split}
\int_{\Rset^N}&(I_\alpha*F(U_{R}^{q}))F(U_{R}^{q})=A_\alpha\sum_{z_j\in Gq}\sum_{z_m\in Gq}\int_{B^j_{2R}}
\int_{B^m_{2R}}K_{jm}(x,y)\dif x\dif y\\
&=A_\alpha\sum_{z_j\in Gq}\int_{B^j_{2R}} \int_{B^j_{2R}}K_{jj}(x,y)\dif x\dif y
+A_\alpha\sum_{z_j\neq z_m\in Gq}\int_{B^j_{2R}} \int_{B^m_{2R}}K_{jm}(x,y)\dif x\dif y
\end{split}
\end{equation*}
where the function $K_{jm}:\Rset^N\times\Rset^N\to \Rset$ is defined by
$$
K_{jm}(x,y):=\frac{F\big((\zeta_R U_{R}^{q})(x-l_GRz_j)\big)F\big((\zeta_R U_{R}^{q})(y-l_GRz_m)\big)}{\abs{x-y}^{N-\alpha}}.
$$
Moreover, we observe that
$
(K_1l_G-4) R\leq \abs{x-y}\leq (K_2l_G+4)R
$
for  any $x\in B^j_{2R}$ and $y\in B^m_{2R}$ with $z_j\neq z_m$.
It then follows that
\begin{align}\label{asymptoticestimateQ}
 &\int_{\Rset^N}(I_\alpha*F(U^q_{R}))F(U^q_{R})\nonumber&\\
  &\geq {\abs{Gq}}\int_{B_{2R}} \int_{B_{2R}}A_\alpha\frac{F\big((\zeta_R U_{S_q})(x)\big) F\big((\zeta_R U_{S_q})(y)\big)}
 {\abs{x-y}^{N-\alpha}}&\\
&\quad+\sum_{z_j\neq z_m\in Gq}\frac{\C}{R^{N-\alpha}}\int_{B^j_{2R}}
F\big((\zeta_R U_{S_{q}})(x-l_GRz_j)\big)\dif x \int_{B^{m}_{2R}}F\big((\zeta_R U_{S_{q}})(y-l_GRz_m)\big)\dif y\nonumber&\\
&={\abs{Gq}}\int_{B_{2R}} \int_{B_{2R}}A_\alpha\frac{F\big((\zeta_R U_{S_{q}})(x)\big)
F\big(\zeta_R U_{S_{q}})(y)\big)}{\abs{x-y}^{N-\alpha}}\dif x\dif y
+\frac{\C}{R^{N-\alpha}}\Big(\int_{B_{2R}}F\big(\zeta_R U_{S_{q}}\big)\Big)^2.\nonumber
\end{align}
We now estimate the first term that
\begin{equation}\label{asymptoticestimateQ1}
\begin{split}
&\int_{B_{2R}}\int_{B_{2R}}A_\alpha\frac{F\big((\zeta_R U_{S_{q}})(x)\big)
F\big(\zeta_R U_{S_{q}})(y)\big)}{\abs{x-y}^{N-\alpha}}\dif x\dif y=\int_{\Rset^N}(I_\alpha*F(\zeta_RU_{S_q})F(\zeta_RU_{S_q})\\
&\geq\int_{\Rset^N}(I_\alpha*F(U_{S_q})F(U_{S_q})-2\int_{\Rset^N}\big(I_\alpha*F(U_{S_q})\big)\big(F(U_{S_q})-F(\zeta_RU_{S_q})\big).
\end{split}
\end{equation}
\resetconstant
In view of the Hardy--Littlewood--Sobolev inequality \cite[Theorem 4.3]{LL}, we deduce that
\begin{align*}
\Big|\int_{\Rset^N} (I_\alpha*F(U_{S_q})& \big(F(U_{S_q})-F(\zeta_RU_{S_q})\big)\Big| &\\
&\leq \C \Big(\int_{\Rset^N} \abs{F(U_{S_q})}^{\frac{2N}{N+\alpha}} \int_{B_R^c}\big(\abs{F(U_{S_q})}+\abs{F(\zeta_RU_{S_q})}\big)^{\frac{2N}{N+\alpha}} \Big)^{\frac{N+\alpha}{2N}}.
\end{align*}
By the assumptions $(F_1)$--$(F_2)$, we see that there exists $C>0$ such that for all $s\in \Rset$
$$\abs{F(s)}\leq C (\abs{s}^{\frac{N+\alpha}{N}}+\abs{s}^{\frac{N+\alpha}{N-2}}),$$
this implies that $\int_{\Rset^N}\abs{F(U_{S_q})}^{\frac{2N}{N+\alpha}}$ is bounded
and
\begin{equation*}
\int_{B_R^c} \abs{F(\zeta_R U_{S_q})}^{\frac{2N}{N+\alpha}}, \;\int_{B_R^c} \abs{F(U_{S_q})}^{\frac{2N}{N+\alpha}}\leq\C \int_{B_R^c}\abs{U_{S_q}}^{2}+\C\int_{B_R^c}\abs{U_{S_q}}^{\frac{2N}{N-2}}.
\end{equation*}
Noticing that $U_{S_q}$ decays exponential at infinity, we conclude that
\begin{equation}\label{exponentialdecayestimate}
2\int_{\Rset^N}(I_\alpha* F(U_{S_q}))\big(F(U_{S_q})-F(\zeta_RU_{S_q})\big)=o(\frac{1}{R^{N-\alpha}}).
\end{equation}
Similarly, by the assumptions $(F_1)$ and $(F_2)$ again we derive that
$\abs{F(t)}\leq C(\abs{t}^2+\abs{t}^\frac{N+\alpha}{N-2})$ for some positive constant $C$,
so that
\begin{equation}\label{asymptestimateQ3}
\frac{1}{R^{N-\alpha}}\Big(\int_{\Rset^N} F(U_{S_q})\Big)^2-\frac{1}{R^{N-\alpha}}\Big(\int_{\Rset^N}F(\zeta_RU_{S_q})\Big)^{2}=o(\frac{1}{R^{N-\alpha}}).
\end{equation}
By combining the above \eqref{asymptoticestimateQ}--\eqref{asymptestimateQ3}, we arrive at an asymptotic lower bound for $Q(U_R^q)$ that
\begin{align}\label{Leqdenominatorestimate}
&\int_{\Rset^N}(I_\alpha*F(U^q_R))F(U_R^q)\nonumber&\\
&\geq \frac{\abs{G}}{\abs{S_q}}\int_{\Rset^N}(I_\alpha*F(U_{S_q}))F(U_{S_q})+\frac{\C}{R^{N-\alpha}}\Big(\int_{B_R}F(\zeta_RU_{S_q})\Big)^2 +o(\frac{1}{R^{N-\alpha}})\nonumber&\\
&=\frac{\abs{G}}{\abs{S_q}}\int_{\Rset^N}(I_\alpha*F(U_{S_q}))F(U_{S_q})+\frac{\C}{R^{N-\alpha}}\Big(\int_{\Rset^N}F(U_{S_q})\Big)^2 +o\big(\frac{1}{R^{N-\alpha}}\big).
\end{align}
This completes the claim.  By Lemma \ref{pathofpohozaev}, there exist a path $\gamma_R\in C([0,+\infty);\cH_G)$ and a unique $t_R\in (0,+\infty)$ such that
 $\mathcal{P}(\gamma_R(t_R))=0$ so that $\gamma_R(t_R)\in \mathcal{P}_G$.
Direct calculations give us that
\begin{equation}\label{boundedtr}
\frac{N-2}{2}t_R^{N-2}\int_{\Rset^N}\abs{\nabla U^q_R}^2+\frac{N}{2}t_R^N \int_{\Rset^N}\abs{U_R^q}^2=\frac{N+\alpha}{2}t_R^{N+\alpha}\int_{\Rset^N}(I_\alpha*F(U^q_R))F(U^q_R).
\end{equation}
On the other hand,  we deduce by integration by parts that
\begin{equation*}
\begin{split}
\int_{\Rset^N}\abs{\nabla U^q_R}^2+\abs{U^q_R}^2&=\sum_{z_j\in Gq}\int_{B_{2R}^j}\abs{\nabla U^q_R}^2+\abs{U^q_R}^2=\frac{\abs{G}}{\abs{S_q}} \int_{B_{2R}}\abs{\nabla \big(\zeta_R U_{S_q}\big)}^2+\abs{\zeta_RU_{S_q}}^2\\
&=\frac{\abs{G}}{\abs{S_q}}\int_{\Rset^N}\zeta_R^2(\abs{\nabla U_{S_q}}^2+\abs{U_{S_q}}^2)
  -\frac{\abs{G}}{\abs{S_q}}\int_{\Rset^N}\zeta_R(\Delta \zeta_R)\abs{U_{S_q}}^2\\
&\leq \frac{\abs{G}}{\abs{S_q}}\int_{\Rset^N}\abs{\nabla U_{S_q}}^2+\abs{U_{S_q}}^2+\frac{\C}{R^2}\int_{B_{2R}\setminus B_R}\abs{U_{S_q}}^2.
\end{split}
\end{equation*}
Again, the exponential decay of the solution $U_{S_q}$ implies that
\begin{equation}\label{Leqnumeratorestimate}
\int_{\Rset^N}\big(\frac{N-2}{2}\abs{\nabla U^q_R}^2+\frac{N}{2} \abs{U^q_R}^2\big)
\leq  \abs{Gq} \int_{\Rset^N}\big(\frac{N-2}{2}\abs{\nabla U_{S_q}}^2+\frac{N}{2} \abs{U_{S_q}}^2\big) +o\big(\frac{1}{R^{N-\alpha}}\big).
\end{equation}
Taking into account of these estimates \eqref{Leqdenominatorestimate} and \eqref{Leqnumeratorestimate}, we then derive from \eqref{boundedtr}
that $t_R$ is uniformly bounded in $R$. We thus deduce an asymptotic bound that
\begin{equation*}
\begin{split}
\mathcal{E}(\gamma_R(t_R))&=\frac{t_R^{N-2}}{2}\int_{\Rset^N}\abs{\nabla U_R^q}+\frac{t_R^{N}}{2}\int_{\Rset^N}\abs{U_R^q}^2-\frac{t_R^{N+\alpha}}{2}\int_{\Rset^N}(I_\alpha*F(U_R^q))F(U_R^q)\\
&\leq \abs{Gq}  \mathcal{E}( \gamma_{U_{S_q}}(t_R))-\frac{C}{R^{N-\alpha}}+o(\frac{1}{R^{N-\alpha}})\\
&\leq\abs{Gq} \max_{t\geq 0} \mathcal{E}(\gamma_{U_{S_q}}(t))-\frac{C}{R^{N-\alpha}}+o(\frac{1}{R^{N-\alpha}})\\
&=\abs{Gq} \mathcal{E}(U_{S_q})-\frac{C}{R^{N-\alpha}}+o(\frac{1}{R^{N-\alpha}}).
\end{split}
\end{equation*}
The desired conclusion $c_G<{\abs{Gq}}c_{S_q}$ follows provided that $R$ is sufficiently large.
We note that the Proposition \ref{energycharacterization} is used in the final equality that $U_{S_q}$ achieves $p_G$. We also mention the Lagrange's theorem on the finite group that $\abs{G}=\abs{Gq}\abs{S_q}.$
\end{proof}

\section{Proof of Theorems}
\label{section5}
\resetconstant
We now prove that, up to translations, the Pohozaev--Palais--Smale sequence constructed in Proposition \ref{LpropPSsequence} has a nontrivial weak limit below
the energy level $c_{G}$.
\begin{lemma}\label{procompactness}
Let $N\geq 2$, $\alpha\in(0,N)$ and $G\in \mathcal{G}_k$ with $1\leq k\leq N$. Assume that the sequence
$(u_n)_{n\in\Nset}\subset \mathcal{H}_G$ satisfies as $n\to\infty$,
$$
\mathcal{E}(u_n)\to c_{G}, \quad \mathcal{E}'(u_n)\to 0 \quad\text{\rm{ strongly in } }\mathcal{H}_G^* \quad \text{ \rm{and} } \mathcal{P}(u_n)\to 0.
$$
Then, there exists a sequence of points $(a_n)_{n\in\Nset}\subset\{0\}\times\mathbb{R}^{N-k}\subset\Rset^N$ such that, up to a subsequence,
$(u_n(\cdot-a_n))_{n\in\Nset}$ converges weakly to $u\in \mathcal{H}_G$.
Moreover,
$$
0<\mathcal{E}(u)\leq c_{G}\qquad\text{\rm{and}}\qquad \mathcal{E}'(u)=0 \quad\text{\rm{in}}\quad \mathcal{H}_G^*.
$$
\end{lemma}
\begin{proof}
The sequence $(u_n)_{n\in\Nset}$ is bounded in $H^1(\Rset^N)$ because we have
\begin{equation*}
 (\alpha+2) A(u_n)+\alpha B(u_n)= 2(N+\alpha)\mathcal{E}(u_n)-2\mathcal{P}(u_n)+o(1).
\end{equation*}
 We claim that there exists $R>0$ such that for some $p\in (\frac{N+\alpha}{N},\frac{N+\alpha}{N-2})$
\begin{equation}\label{mainclaim}
\liminf_{n\to\infty}\int_{O_R}\abs{u_n}^{\frac{2Np}{N+\alpha}}>0.
\end{equation}
Here $O_R:=[-R,R]^k\times\Rset^{N-k}$. We acknowledge this claim \eqref{mainclaim} temporarily and first finish the proof of the lemma.
Taking a function $\varphi\in C^\infty_c(\Rset^N)$ such that $\varphi=1$ on $O_R$, $\supp \varphi \subset  O_{3R/2}$, $0\leq \varphi\leq 1$ on $\Rset^N$ and
$\nabla \varphi \in L^\infty(\Rset^N)$, we then deduce from \cite[Lemma 1.1]{Lions1984} (see also  \cite[Lemma 2.3]{MVJFA} or \cite[ Lemma 1.21]{Willem1996})
that there exists a constant $K_G>0$ such that
\begin{equation*}
\begin{split}
\int_{O_R}\abs{u_n}^{\frac{2Np}{N+\alpha}}
&\leq \int_{\Rset^N}\abs{\varphi u_n}^{\frac{2Np}{N+\alpha}}\\
&\leq \C \Big(\sup_{a\in \Rset^N}\int_{B_{R/2}(a)}\abs{\varphi u_n}^{\frac{2Np}{N+\alpha}}\Big)^{1-\frac{N+\alpha}{Np}}\int_{\Rset^N}
\abs{\nabla(\varphi u_n)}^2+\abs{\varphi u_n}^2\\
&\leq \C \Big(\sup_{a\in \{0\}\times\Rset^{N-k}}\int_{B_{K_GR}(a)}\abs{u_n}^{\frac{2Np}{N+\alpha}}
\Big)^{1-\frac{N+\alpha}{Np}}\int_{\Rset^N}\abs{\nabla u_n}^2+\abs{u_n}^2.
\end{split}
\end{equation*}
Hence, there exists a sequence of points $(a_n)_{n\in\Nset}\subset\{0\}\times\Rset^{N-k}$ such that
$$
\liminf_{n\to\infty}\int_{B_{K_GR}(a_n)}\abs{u_n}^{\frac{2Np}{N+\alpha}}>0.
$$
We then conclude by a similar argument as in \cite[Proposition 2.2]{MVTAMS} that, up to translations and a subsequence, $(u_n)_{n\in\Nset}$ converges weakly to some function $u\in \mathcal{H}_G\setminus\{0\}$ and $\mathcal{E}'(u)=0$ in $\mathcal{H}_G^*$. We thus conclude that $\mathcal{P}(u)=0$.
By the weakly lower semicontinuity of the norm, we are left that $\cE(u)\leq c_G$. In fact, we have
\begin{equation*}
\begin{split}
c_{G}
&=\lim_{n\to\infty} \Big(\mathcal{E}(u_n)-\frac{1}{N+\alpha}\mathcal{P}(u_n)\Big)\\
&=\lim_{n\to\infty} \frac{\alpha}{2(N+\alpha)} \int_{\Rset^N}\abs{\nabla u_n}^2+\frac{\alpha}{2(N+\alpha)}\int_{\Rset^N}\abs{u_n}^2\\
&\geq \frac{\alpha}{2(N+\alpha)}\int_{\Rset^N}\abs{\nabla u}^2+\abs{u}^2=\mathcal{E}(u)-\frac{1}{N+\alpha}\mathcal{P}(u)=\mathcal{E}(u).
\end{split}
\end{equation*}

We now return to finish the proof of the claim \eqref{mainclaim}. Up to a subsequence, there exists $\Lambda\in(0,+\infty)$ such that
$$
\lim_{n\to\infty}\int_{\Rset^N} \abs{u_n}^{\frac{2Np}{N+\alpha}}=\Lambda.
$$
Note that the sequence $(u_n)_{n\in\Nset}$ is bounded in $H^1(\Rset^N)$. The Rellich--Kondrachov embedding theorem guarantees the finiteness of $\Lambda$.
We suppose that $\Lambda=0$. By $(F_1)$ and $(F_2)$, we deduce that for any $\varepsilon>0$, there exists $C_\varepsilon$ such that
\begin{equation*}\label{assumptionsof1.1}
\abs{F(u)}\leq \varepsilon (\abs{u}^{\frac{N+\alpha}{N}}+ \abs{u}^{\frac{N+\alpha}{N-2}}) +C_\varepsilon \abs{u}^p.
\end{equation*}
By the Hardy--Littlewood--Sobolev inequality and by the Rellich--Kondrachov embedding theorem, we derive that
$$ \lim_{n\to\infty}\int_{\Rset^N}(I_\alpha*F(u_n))F(u_n)=0.$$
Inserting this into
$\lim_{n\to\infty}\cP(u_n)=0$, we deduce $\lim_{n\to\infty}\|u_n\|= 0$, then a contradiction follows that
$
c_{G}=\lim_{n\to\infty}\mathcal{E}(u_n)=0.
$

Since $G=\langle\cS\rangle$, in the subsequent parts of this paper, we
assume each of the hyperplane fixed by the reflection $s\in \cS$ passing through the origin and $\cF$
always refer to the  strict fundamental domain under the group action $G\times 1_{N-k}$, see Appendix. Without loss of generality, we can fix $\cF$ to be the infinite cone formed by $H_s$ for all $s\in \cS$. Let $\partial^i \mathcal{F}$ denote its $k-i$ dimensional facets of the polyhedral cone $\mathcal{F}$ for $i=0,1,\cdots,k-1$, here $\partial^{0}\mathcal{F}$ refers to the interior of $\mathcal{F}$ for convenience.
We claim that for any $r>0$,  $$\liminf_{n\to\infty}\sup_{y\in\Rset^N}\int_{B_r(y)}\abs{u_n}^{\frac{2Np}{N+\alpha}}>0.$$
Otherwise, by Lions' Lemma (e.g., \cite{Willem1996}) the sequence $(u_n)_{n\in\Nset}\subset H^1(\Rset^N)$ converges strongly to zero in $L^{\frac{2Np}{N+\alpha}}(\Rset^N)$ as $n\to\infty$. Then a similar argument as above implies $c_G=0$ again. In other words, the sequence $(u_n)_{n\in\Nset}$ is non-vanishing.  According to the P. -L. Lions' concentration-compactness principle \cite{Lions1984}, up to a subsequence, there are still two possibilities: compactness and dichotomy.\\
\textbf{Compactness}: there exists $(x_n)_{n\in\Nset}\subset \Rset^N$ such that
\begin{itemize}
\item [] $$ \forall \varepsilon>0, \exists R>0 \text{ such that }  \liminf_{n\to\infty}\int_{B_R(x_n)}\abs{u_n}^{\frac{2Np}{N+\alpha}}\geq \Lambda-\varepsilon.$$
\end{itemize}
Without loss generality, we assume that $(x_n)_{n\in\Nset}\subset \cF$. We then easily deduce that there exists $M>0$ such that
$
\sum_{j=1}^k\abs{x_n^j}\leq M
$
for all $x_n=(x_n^1,x_n^2,\cdots, x_n^N)$.
Otherwise, for largely $n$, there exists some $\mathcal{R}\in G$ such that
$
B_{R}(x_n)\cap \mathcal{R}(B_{R}(x_n))=\emptyset.
$
In fact, if $x_n\in \partial^0\cF$, we can choose any $\mathcal{R}\in \mathcal{S}$; if $x_n\in \partial^i \cF\subset \{\cap H_s \st s\in \mathcal{S}_i\subset \mathcal{S}\}\times\Rset^{N-k}$, then we take a reflection $\mathcal{R}\in \mathcal{S}\setminus\mathcal{S}_i$.
It then follows that
$$
\liminf_{n\to\infty}\int_{(B_{R}(\mathcal{R} x_n))}\abs{u_n}^{\frac{2Np}{N+\alpha}}=\liminf_{n\to\infty}\int_{B_{R}(x_n)}\abs{u_n}^{\frac{2Np}{N+\alpha}}\geq \Lambda-\varepsilon.
$$
This would lead to a contradiction once $\varepsilon$ is sufficiently small since
$$
\liminf_{n\to\infty}\int_{\Rset^N}\abs{u_n}^{\frac{2Np}{N+\alpha}}
\geq \liminf_{n\to\infty}\int_{B_{R}(x_n)}\abs{u_n}^{\frac{2Np}{N+\alpha}}+\liminf_{n\to\infty}\int_{\mathcal{R}(B_{R}(x_n))}\abs{u_n}^{\frac{4N}{N+\alpha}}
\geq 2\Lambda-2\varepsilon.
$$
\textbf{Dichotomy}: there exists  $(y_n)_{n\in\Nset}\subset \Rset^N$ and $\delta\in(0,1)$ such that $\forall \varepsilon>0$, there exists $R_0>0$ such that for any $r,r'\geq R_0$, there holds
$$
\limsup_{n\to\infty}\bigg|\int_{B_r(y_n)}\abs{u_n}^{\frac{2Np}{N+\alpha}}-\delta\Lambda\bigg|
+\limsup_{n\to\infty}\bigg|\int_{B^c_{r'}(y_n)}\abs{u_n}^{\frac{2Np}{N+\alpha}}-(1-\delta)\Lambda\bigg|\leq \varepsilon.
$$
For this situation, we also assume by the symmetry that $(y_n)_{n\in\Nset}\subset \cF$.

\noindent \textbf{Case 1}: If there exists $R_1>0$ such that
$
(y_n)_{n\in\Nset}\subset B_{R_1}\cap \cF
$,
we then conclude the claim \eqref{mainclaim} straight away since $\delta\Lambda>0$.

For the subsequent proofs, we always assume the cut-off function $\xi\in[0,1]$
satisfies $\xi=0$ for $s\leq 1$ or $s\geq 4$, $\xi(s)=1$ for $2\leq s\leq 3$ and $\abs{\xi'(s)}\leq 2$; the cut-off function $\eta\in[0,1]$
satisfies that $\eta(t)=1$ for $t\in[0,2]$, $\eta(t)=0$ whenever $t\geq 3 $ and $\abs{\eta'(t)}\leq 2$.
Suppose that the claim \eqref{mainclaim} is not true, by a diagonal process and up to a subsequence we may
choose $\varepsilon_n\to 0$, $r_n\to+\infty $ as $n\to\infty$ and $r_n'=4r_n$ such that
$$\bigg|\int_{B_{r_n}(y_n)}\abs{u_n}^{\frac{2Np}{N+\alpha}}-\delta\Lambda\bigg|
+\bigg|\int_{B^c_{r'_n}(y_n)}\abs{u_n}^{\frac{2Np}{N+\alpha}}-(1-\delta)\Lambda\bigg|\leq \varepsilon_n.
$$
The subsequent proofs will be completed by splitting into several cases according to the asymptotic distance between the sequence $(y_n)_{n\in\Nset}$ and the boundary $\partial\cF$.\\
\textbf{Case 2}: up to a subsequence, there exists some $i\in\{1,\cdots,k-1\}$ such that $$\lim_{n\to\infty}\dist(y_n, \partial^i \cF)=0 \text{ and }
\lim_{n\to\infty}\dist(y_n, \partial^\ell \cF)=0 \text{ for all } \ell>i.$$

In this situation, we fix a unit vector $q\in \partial^i\cF\cap \partial B_1$ and, up to a subsequence, we may assume that $y_n=\nu_nq$ with $\nu_n\to+\infty$ for each subcases such that
$\abs{x-y}\geq 4r_n$ for $x\in B_{3r_n}(z_1)$ and $y\in B_{3r_n}(z_2)$ where $z_1\neq z_2\in Gy_n$. We define a test function such that
$$
\phi_n^{q}(x)=\sum_{z\in Gy_n}\xi(\abs{x-z}/r_n)u_n(x)=\frac{1}{\abs{S_q}}\sum_{g\in G}\xi(\abs{x-gy_n}/r_n)u_n(x).
$$
It then follows that, $g\diamond \phi_n^q=\psi(g)\phi_n^q$ for any $g\in G$ so that $\phi_n^q\in \mathcal{H}_G$.
Testing the limiting equality that  $\mathcal{E}'(u_n)\to 0$ strongly in $\mathcal{H}_G^*$ as $n\to\infty$ against the function $\phi_n^{q}$, we reach that,
$$
\int_{\Rset^N}\nabla u_n\nabla \phi_n^{q}+u_n\phi_n^{q}=\int_{\Rset^N}(I_\alpha*F(u_n))f(u_n)\phi_n^{q}+o(1).
$$
Taking into  account the growth assumption $(F_1)$ and $(F_2)$ and combining the boundedness of the sequence $(u_n)_{n\in\Nset}\subset\mathcal{H}_G$, the Hardy--Littlewood--Sobolev inequality and the fact that
\begin{equation}\label{outerinvanishing}
\int_{B_{4r_n}(y_n)\setminus B_{r_n}(y_n)}\abs{u_n}^{\frac{2Np}{N+\alpha}}\leq 2\varepsilon_n,
\end{equation}
we then conclude that
\begin{equation}\label{outerinvanishingnorm}
\int_{B_{3r_n}(y_n)\setminus B_{2r_n}(y_n)}\abs{\nabla u_n}^2+\abs{u_n}^2=o(1).
\end{equation}

Define for each $z_j\in G{y_n}$ that
$$
v_n^j(x)=\eta(\abs{x-z_j}/r_n)u_n(x) \; \text{ and }\; w_n=u_n-\sum v_n^j.
$$
We clearly see that $\supp v_n^j\subset B_{3r_n}(z_j)$  for every $1\leq j\leq \abs{G{q}}$. On the other hand,
Lemma \ref{lemmainduction} states that the isotropic group $S_q\in \mathcal{G}_i$ and $S_q=\langle\cS_i\rangle$ for some $\cS_i\subset \cS$ with $\abs{\cS_i}=i$, which leads to $sq=q$ for $s\in \cS_i$.
Let $z_j=g_jy_n$ for some $g_j\in G$, we then deduce for every $s\in \cS_i$,
$$g_jsg_j^{-1}\diamond v_n^j =\eta(\abs{g_jsg_j^{-1}x-z_j}/r_n)u_n(g_jsg_j^{-1}x)=\psi(g_jsg_j^{-1})v_n^j=\psi(s) v_n^j.$$
This means that $v_n^j\in \mathcal{H}_{S_{q}}$ for all $1\leq j\leq \abs{Gq}$ since the group $g_jS_{q}g_j^{-1}$ is isomorphic to $S_q$  for every $j$.
For simplicity, we denote $v_n=\sum_j v_n^j$, so that
$$v_n=\sum_{z_j\in G{\nu_n q}} \eta(\abs{x-z_j}/r_n)u_n=\frac{1}{\abs{S_q}}\sum_{g\in G}\eta(\abs{x-g\nu_nq}/r_n)u_n.$$
We thus have $v_n\in \mathcal{H}_G$ since $g\diamond v_n=\psi(g)v_n$ for any $g\in G$.
It follows that $w_n=u_n-v_n\in \mathcal{H}_G$.
We denote $B_r^j=B_r(z_j)$ with $r>0$ for $z_j\in G{\nu_nq}$, and take $\mathcal{B}_r =\cup_{z_j} B_r^j$ for simplicity.
 By symmetry and by the choice of the cut-off function, we see from \eqref{outerinvanishingnorm} that
\begin{align}
\int_{\Rset^N}&\nabla v_n\nabla w_n+v_nw_n=\sum_{z_j\in G{\nu_nq}}\int_{B_{3r_n}(z_j)\setminus B_{2r_n}(z_j) }\nabla{v_n^j}\nabla w_n+v_n^jw_n&\nonumber\\
&\leq \frac{\C}{r_n}\int_{\Rset^N}\abs{\nabla u_n}^2+\abs{u_n}^2+\C\int_{B_{3r_n}(y_n)\setminus B_{2r_n}(y_n)}\abs{\nabla u_n}^2+\abs{u_n}^2=o(1).\label{gradientterm}
\end{align}
Note that $\abs{x-y}\geq 4r_n$ for any $x\in \supp v_n^j$, $y\in \supp v_n^m$ provided that $j\neq m$.
Taking $\beta\in(\alpha, N)$ such that $ \beta< N\frac{2+\alpha}{N-2}$, we then deduce that for $1\leq j\neq m\leq {\abs{Gq}}$,
\begin{align}
\int_{\Rset^N}(I_\alpha*F(v^j_n))F(v^m_n)
&\leq \frac{\C}{r_n^{\beta-\alpha}}\int_{\Rset^N}(I_\beta*F(u_n))F(u_n)&\nonumber\\
&\leq  \frac{\C}{r_n^{\beta-\alpha}}\Big(\int_{\Rset^N}\abs{F(u_n)}^{\frac{2N}{N+\beta}}\Big)^{\frac{N+\beta}{N}}=o(1). \label{mixtermconvolution}
\end{align}
This is a consequence of Rellich--Kondrachov embedding theorem. In fact, our choice of $\beta$ ensures the exponent $\frac{2N}{N+\beta}\frac{N+\alpha}{N-2}\in(2,\frac{2N}{N-2})$,  we then deduce by $(F_1)$ and $(F_2)$ that
$$\int_{\Rset^N}\abs{F(u_n)}^{\frac{2N}{N+\beta}}\leq \C \int_{\Rset^N} \abs{u_n}^2+\abs{u_n}^{\frac{2N}{N+\beta}\frac{N+\alpha}{N-2}}
\leq \C \big(\|u_n\|^2+\|u_n\|^{\frac{2N}{N+\beta}\frac{N+\alpha}{N-2}}\big).$$
Here we recall that $\frac{N+\alpha}{N-2}=\bar{p}$ and $\frac{2N}{N-2}=\infty$ for $N=2$.
\resetconstant

We shall prove the following two decompositions on the functionals:
\begin{equation}\label{energydecomposition}
\mathcal{E}(u_n)=\sum_{1\leq j\leq \abs{Gq}}\mathcal{E}(v_n^j)+\mathcal{E}(w_n)+o(1),
\end{equation}
\begin{equation}\label{pohozaevdecomposition}
\mathcal{P}(u_n)=\sum_{1\leq j\leq \abs{Gq}}\mathcal{P}(v_n^j)+\mathcal{P}(w_n)+o(1).
\end{equation}
For brevity, we set
$$
J_n(x,y)=A_\alpha\bigg(\frac{F(u_n(x))F(u_n(y))}{\abs{x-y}^{N-\alpha}}-\frac{F(v_n(x))F(v_n(y))}{\abs{x-y}^{N-\alpha}}
-\frac{F(w_n(x))F(w_n(y))}{\abs{x-y}^{N-\alpha}}\bigg).
$$
By direct calculations, we expand $\cE(u_n)$ in terms of $\cE(v_n^j)$ and $\cE(w_n)$ that
\begin{equation*}
\begin{split}
&\mathcal{E}(u_n)=\mathcal{E}(v_n)+\mathcal{E}(w_n)+\int_{\Rset^N}\nabla v_n\nabla w_n+v_nw_n -\frac{1}{2} \int_{\Rset^N}\int_{\Rset^N}J_n(x,y)\dif x\dif y\\
&\quad\;\quad\;=\sum_{1\leq j\leq \abs{Gq}}\mathcal{E}(v^j_n)+\mathcal{E}(w_n)-\frac{1}{2}\sum_{1\leq j\neq m\leq  \abs{Gq}}\int_{\Rset^N}(I_\alpha*F(v^j_n))F(v^m_n)\\
 &\quad\;\quad\qquad  +\int_{\Rset^N}\nabla v_n\nabla w_n+v_nw_n -\frac{1}{2}\int_{\Rset^N}\int_{\Rset^N}J_n(x,y)\dif x\dif y
\end{split}
\end{equation*}
Similarly, we can also obtain the following  expansion,
\begin{alignat}{4}
&\mathcal{P}(u_n)
=\sum_{1\leq j\leq \abs{Gq}}\mathcal{P}(v^j_n)+\mathcal{P}(w_n)-\frac{N+\alpha}{2}\sum_{ j\neq m}\int_{\Rset^N}(I_\alpha*F(v^j_n))F(v^m_n)\nonumber\\
 &\qquad\;\qquad\;  +(N-2)\int_{\Rset^N}\nabla v_n\nabla w_n+N\int_{\Rset^N}v_nw_n -\frac{N+\alpha}{2}\int_{\Rset^N}\int_{\Rset^N}J_n(x,y)\dif x\dif y\nonumber.
\end{alignat}
In view of the estimates \eqref{gradientterm} and \eqref{mixtermconvolution},  the conclusions  \eqref{energydecomposition} and \eqref{pohozaevdecomposition} follows once we  prove
\begin{equation}\label{mixterm}
\int_{\Rset^N} \int_{\Rset^N}J_n(x,y)\dif x\dif y=o(1).
\end{equation}
On one hand, similar manipulations as in \eqref{mixtermconvolution}   bring us that
\begin{equation*}
\begin{split}
 \int_{\mathcal{B}_{2r_n}}\int_{\mathcal{B}^c_{3r_n}} \frac{F(\varphi_n(x))F(\varphi_n(y))}{\abs{x-y}^{N-\alpha}}\dif x\dif y
 &\leq \frac{\C}{r_n^{\beta-\alpha}}\Big(\int_{\Rset^N}\abs{F(\varphi_n)}^{\frac{2N}{N+\beta}}\Big)^{\frac{N+\beta}{N}}\\
 &\leq \frac{\C}{r_n^{\beta-\alpha}} \big(\|u_n\|^2+\|u_n\|^{\frac{2N}{N+\beta}\frac{N+\alpha}{N-2}}\big)=o(1).
 \end{split}
\end{equation*}
Here $\varphi_n$ can be any sequence among $u_n$, $v_n$ and $w_n$. We thus conclude that $$\int_{\mathcal{B}_{2r_n}}\int_{\mathcal{B}^c_{3r_n}} J_n(x,y)\dif x\dif y=o(1).$$
On the other hand, by the Hardy--Littlewood--Sobolev inequality and by the Rellich--Kondrachov embedding theorem, we infer from \eqref{outerinvanishing} and \eqref{outerinvanishingnorm} that for any $\varphi_n$ amongst $u_n$, $v_n$ and $w_n$
\begin{equation*}
\begin{split}
&\int_{\mathcal{B}_{3r_n}\setminus\mathcal{B}_{2r_n}} \int_{\Rset^N}\frac{F(\varphi_n(x))F(\varphi_n(y))}{\abs{x-y}^{N-\alpha}}\dif x\dif y\\
&\leq \C\bigg(\int_{\Rset^N}\abs{F(\varphi_n)}^{\frac{2N}{N+\alpha}}\bigg)^{\frac{N+\alpha}{2N}}\bigg(\int_{B_{3r_n}(y_n)\setminus B_{2r_n}(y_n)}\abs{F(\varphi_n)}^{\frac{2N}{N+\alpha}}\bigg)^{\frac{N+\alpha}{2N}}\\
&\leq  \C\Big(\|u_n\|^{2}+\|u_n\|^{\frac{N+\alpha}{N-2}\frac{2N}{N+\alpha}}\Big)^{\frac{N+\alpha}{2N}} \\
&\qquad \times \bigg(\int_{B_{3r_n}(y_n)\setminus B_{2r_n}(y_n)}  \abs{u_n}^{\frac{N+\alpha}{N-2}\frac{2N}{N+\alpha}}+  \abs{u_n}^2+C_\varepsilon  \abs{u_n}^{\frac{2Np}{N+\alpha}}\bigg)^{\frac{N+\alpha}{2N}} =o(1).
\end{split}
\end{equation*}
We then obtain
$$\int_{\mathcal{B}_{3r_n}\setminus\mathcal{B}_{2r_n}} \int_{\Rset^N} J_n(x,y)\dif x\dif y=o(1).$$
The very similar analyses as above also indicate that
\begin{equation*}
 \int_{\mathcal{B}^c_{3r_n}}\int_{\mathcal{B}_{2r_n}} J_n(x,y)\dif x\dif y=o(1)\;\; \text{ and }\;\; \int_{\Rset^N}\int_{\mathcal{B}_{3r_n}\setminus\mathcal{B}_{2r_n}}J_n(x,y)\dif x\dif y=o(1).
\end{equation*}
We are now ready to complete the final estimate on the convolution term \eqref{mixterm} that,
\begin{equation*}\label{mixedterm}
\begin{split}
\int_{\Rset^N}\int_{\Rset^N}&J_n(x,y)\dif x\dif y
=\int_{\mathcal{B}_{2r_n}\cup (\mathcal{B}_{3r_n}\setminus \mathcal{B}_{2r_n})\cup  \mathcal{B}^{c}_{3r_n}}
\int_{\mathcal{B}_{2r_n}\cup (\mathcal{B}_{3r_n}\setminus \mathcal{B}_{2r_n})\cup \mathcal{B}^{c}_{3r_n}}J_n(x,y)\dif x\dif y\\
&\leq \int_{\Rset^N}\int_{\mathcal{B}_{3r_n}\setminus\mathcal{B}_{2r_n}}J_n(x,y)\dif x\dif y
+\int_{\mathcal{B}_{2r_n}}\int_{\mathcal{B}^{c}_{3r_n}}J_n(x,y)\dif x\dif y\\
&\quad  +\int_{\mathcal{B}_{3r_n}\setminus\mathcal{B}_{2r_n}}\int_{\Rset^N}J_n(x,y)\dif x\dif y
+\int_{\mathcal{B}^{c}_{3r_n}}\int_{\mathcal{B}_{2r_n}}J_n(x,y)\dif x\dif y=o(1).
\end{split}
\end{equation*}

By the Rellich--Kondrachov embedding theorem and the fact that
$$\liminf_{n\to\infty}\int_{\Rset^N}\abs{v^j_n}^{\frac{2Np}{N+\alpha}}
\geq\liminf_{n\to\infty}\int_{B^j_{r_n}}\abs{u_n}^{\frac{2Np}{N+\alpha}}\geq \delta \Lambda >0,$$
for each $1\leq j\leq \abs{Gq}$, we then deduce that
 $\liminf_{n\to\infty}\|v_n^j\|>0$. Note that $$\mathcal{P}(u_n)=\sum\cP(v_n^j)+\cP(w_n)=o(1).$$
 We claim that $\cP(w_n)=o(1)$, so that $\cP(v_n^j)=o(1)$ for every  $1\leq j\leq \abs{Gq}$, which bring us that $\lim_{n\to\infty}Q(v^j_n)>0$ for each $j$. By Lemma \ref{pathofpohozaev}, there exists $s_n^j\in (0,+\infty)$ such that $ v_n^j(\frac{\cdot}{s_n^j})\in \mathcal{P}_{S_q}$. In particular, $\cP(v_n^j)=o(1)$ implies
$\lim_{n\to\infty} s_n^j=1$.
On the other hand,
$$\liminf_{n\to\infty}\cE(w_n)=\liminf_{n\to\infty}\cE(w_n)-\frac{1}{N+\alpha}\lim_{n\to\infty}\cP(w_n)\geq 0.$$
We thus deduce
\begin{align*}
c_G=\lim_{n\to\infty}\sum\mathcal{E}(v_n^j)+\lim_{n\to\infty}\mathcal{E}(w_n)\geq\lim_{n\to\infty}
\sum\mathcal{E}(v_n^j(\frac{\cdot}{s_n^j}))\geq   \abs{Gq}c_{S_q} \geq c_{G}^*.
\end{align*}
This contradicts with the proposition \ref{energystrictinequality} that $c_G<c_{G}^*$.

We now set out to complete the claim that $\cP(w_n)=o(1).$ In fact, since the sequence $(u_n)_{n\in\Nset}$ is bounded, we then easily derive the boundedness of $\cP(v^j_n)$ and $\cP(w_n)$. Up to a subsequence, we assume both  $\cP(w_n)$ and  $\cP(w_n)$ take finite limit values. If $\lim_{n\to\infty}\cP(w_n)>0$, then $\lim_{n\to\infty}\cP(v_n^j)<0$. By Lemma \ref{pathofpohozaev}, there exist $t_n\in(0,1]$ such that $v_n^j(\frac{\cdot}{t_n})\in \cP_{S_q}$. Namely,
$$\frac{t_n^{N-2}}{2}\int_{\Rset^N}\abs{\nabla v^j_n}^2+\frac{t_n^N}{2}\int_{\Rset^N}\abs{v^j_n}^2=\frac{t_n^{N+\alpha}}{2}\int_{\Rset^N}(I_\alpha*F(v^j_n))F(v^j_n).$$
In view of the fact that $\lim_{n\to\infty}\cP(v_n^j)<0$ and the fact that $\cE(w_n)-\frac{1}{N+\alpha}\cP(w_n)\geq 0$, we thus conclude that
\begin{alignat*}{3}
c_G&=\cE(u_n)-\frac{1}{N+\alpha}\cP(u_n)+o(1)
\geq \sum\cE(v_n^j)-\frac{1}{N+\alpha}\cP(v_n^j)+o(1)\nonumber&\\
&=\sum \Big(\frac{1}{2}-\frac{N-2}{2(N+\alpha)}\Big)\int_{\Rset^N}\abs{\nabla v_n^j}^2+\Big(\frac{1}{2}-\frac{N}{2(N+\alpha)}\Big)\int_{\Rset^N}\abs{v_n^j}^2+o(1)&\nonumber\\
&\geq \sum t_n^{N-2} \Big(\frac{1}{2}-\frac{N-2}{2(N+\alpha)}\Big)\int_{\Rset^N}\abs{\nabla v_n^j}^2+ t_n^N\Big(\frac{1}{2}-\frac{N}{2(N+\alpha)}\Big)\int_{\Rset^N}\abs{v_n^j}^2+o(1)&\nonumber\\
&=\sum \cE\big(v_n^j(\frac{\cdot}{t_n})\big)-\cP\big(v_n^j(\frac{\cdot}{t_n})\big)+o(1)\geq \abs{Gq}c_{S_q}+o(1).
\end{alignat*}
This is a contradiction.
For the other case that $\lim_{n\to\infty} \cP(w_n)<0$, we conclude that there exists $t_n\leq 1$ such that $w_n(\frac{\cdot}{t_n})\in \cP_G$. We therefore deduce
\begin{alignat*}{3}
c_G&=\cE(u_n)-\frac{1}{N+\alpha}\cP(u_n)+o(1)\geq \cE(w_n)-\frac{1}{N+\alpha}\cP(w_n)+o(1)&\nonumber\\
&\geq\cE(w_n(\frac{\cdot}{t_n}))-\cP(w_n(\frac{\cdot}{t_n}))+o(1)\geq c_G+o(1).
\end{alignat*}
This leads to $\lim_{n\to\infty} t_n=1$ so that  $\lim_{n\to\infty}\cE(w_n)=c_G$ and $\lim_{n\to\infty}\cE(v^j_n)=0$.
We thus deduce that for each $j$
 $$\lim_{n\to\infty}\cP(v^j_n)=\lim_{n\to\infty}\cP(v^j_n)-(N+\alpha)\lim_{n\to\infty}\cE(v^j_n)
 \leq -\frac{\alpha}{2}\lim_{n\to\infty}\|v_n^j\|^2<0.$$
 A contradiction again.

\textbf{Case 3}: up to a subsequence, for any $\ell\geq 1$,  $\lim_{n\to\infty}\dist(y_n,\partial^\ell \cF)=+\infty.$

In such case, up to a subsequence, we may assume for $z_i\neq z_j\in Gy_n$, there holds $\abs{x-y}\geq 4r_n$ for any $x\in B_{3r_n}(z_i)$ and $y\in B_{3r_n}(z_j)$.
We take the test function that
$$\phi_n^{G}=\sum_{g\in G}\xi(\abs{x-gy_n}/r_n)u_n.$$
We conclude that $\phi_n^G\in \mathcal{H}_G$ since for every $\bar{g}\in G$,
$$\bar{g}\diamond \phi_n^{G}=\sum_{g\in G}\xi(\abs{x-g\bar{g}y_n}/r_n)u_n(\bar{g}^{-1}x)=\psi(\bar{g})\phi_n^G.$$
By repeating the arguments as above, we conclude that
$$
\int_{B_{3r_n}(y_n)\setminus B_{2r_n}(y_n)}\abs{\nabla u_n}^2+\abs{u_n}^2=o(1).
$$
Similarly, we define for each $z_j\in G{y_n}$ that
$$
v_n^i(x)=\eta(\abs{x-z_j}/r_n)u_n(x) \; \text{ and }\; w_n=u_n-\sum v_n^j.
$$
Since $\supp v_n^j\subset B_{3r_n}(z_j)$ for each $j$, we deduce that $v_n^j\in H^1(\Rset^N)$. It can also be checked that $w_n\in \mathcal{H}_G$.
Similarly, we have the decompositions that
\begin{equation}\label{48energydecomposition}
\mathcal{E}(u_n)=\sum \mathcal{E}(v_n^j)+\mathcal{E}(w_n)+o(1), \text{ and } \, \cP(u_n)=\sum\cP(v_n^j)+\cP(w_n)+o(1).
 \end{equation}
Again by the Rellich--Kondrachov embedding theorem, we observe that for all $z_j\in G{y_n}$,
\begin{equation*}\label{48positivity}
\lim_{n\to\infty}\int_{\Rset^N}\abs{\nabla v^j_n}^2+\abs{v_n^j}^2>0.
\end{equation*}
We conclude that $\cP(v_n^j)=o(1)$ and $\cP(w_n)=o(1)$ so that
there exists $s_n^j\in(0,+\infty)$ such that $v_n^j(\frac{\cdot}{s_n^j})\in \cP$ with
$\lim_{n\to\infty} s_n^j=1$.
We therefrom deduce by the above \eqref{48energydecomposition}  that
\begin{align*}
c_G=\lim_{n\to\infty}\mathcal{E}(u_n)&=\sum \lim_{n\to\infty}\mathcal{E}(v_n^j)+\lim_{n\to\infty}\mathcal{E}(w_n)\geq
\sum \lim_{n\to\infty}\mathcal{E}(v_n^j(\frac{\cdot}{s_n^j}))
\geq \abs{G}c_0.
\end{align*}
This is in contradiction to the proposition \ref{energystrictinequality} that $c_G<c_{G}^*<\abs{G}c_0$.
\end{proof}

\vskip .1in

\noindent\textbf{Proof of Theorem \ref{thm1.1}.}  By induction, we suppose Proposition \ref{energycharacterization} holds provided that the rank of Coxeter group is no more than $k-1$. We consider the mountain pass problem \eqref{minimizationprob} in which the Coxeter group $G$ is of rank $k$. By Proposition \ref{LpropPSsequence},
there exists a Pohozaev--Palais--Smale sequence $(u_n)_{n\in\Nset}$ in $\mathcal{H}_G$ at the energy level $c_G$. We then obtain a nontrivial weak solution $u\in \mathcal{H}_G$ with $\cE(u)\leq c_G$
by applying Proposition \ref{energystrictinequality} and Lemma \ref{procompactness}. This completes  Lemma \ref{energycharacterization} that $c_G=p_G=m_G=\cE(u)$, which in turn yields a strong convergence of $(u_n)_{n\in\Nset}$ in $H^1(\Rset^N)$. By the symmetric criticality principle \cite[Theorem 1.28]{Willem1996}),
we see that $u$ is a critical point of $\mathcal{E}$ in $H^1(\Rset^N)$, so that $u$ is a $G$-groundstate.
The regularity $u\in C^2(\Rset^N)$ follows from Theorem \ref{thm1.2}.

We next show that $u$ keeps a constant sign on the interior of the strict fundamental domain $\mathcal{F}$.
Our proofs essentially follow the arguments of \cite[Proposition 2.6]{GhimentiVanSchaftingen2016}, see also \cite{XiaWang2019,XiaWang2018,XiaZhang2020}.

\noindent\textbf{Proof of Theorem \ref{thm1.3}.}
Taking the subspace
$$
H^1_0(\mathcal{F})=\{v\in H^1(\mathcal{F})\st v=0 \text{ almost everywhere in } \Rset^N\setminus {\rm{int}}(\mathcal{F})\},
$$
we define two functionals $\mathcal{E}_0,\cP_0\in C^1(H^1_0(\mathcal{F}),\mathbb R)$ such that for $v\in H^1_0(\mathcal{F})$,
$$
\cE_0(v)=\frac{\abs{G}}{2} \int_{\mathcal{F}}\abs{\nabla v}^2+\abs{v}^2-\frac{1}{2}\int_{\mathcal{F}}\int_{\mathcal{F}}K_\alpha(x,y)F(v(x))
F(v(y))\dif x\dif y
$$
and
$$
\cP_0(v)=\frac{(N-2)\abs{G}}{2}\int_{\mathcal{F}}\abs{\nabla v}^2+\frac{N\abs{G}}{2}\int_{\mathcal{F}}\abs{v}^2-\frac{N+\alpha}{2}\int_{\mathcal{F}}\int_{\mathcal{F}}K_\alpha(x,y)F(v(x))
F(v(y))\dif x\dif y
$$
with the kernel $K_\alpha:\mathcal{F}\times \mathcal{F}\to \mathbb R$ being defined by
$
K_\alpha(x,y)=\sum_{g_i,g_j\in G}A_\alpha \abs{g_i x-g_j y}^{\alpha-N}
$.
We observe that $\chi_{\mathcal{F}} u\in H^1_0(\mathcal{F})$ for $u\in \mathcal{H}_G$ where $\chi_{\mathcal{F}}$ is the characteristic function of the conic region $\mathcal{F}$.  Conversely, for each $v\in H^1_0(\mathcal{F})$, we construct a $G$-symmetry function $\mathcal{U}(v):\Rset^N\to \Rset$ such that
$$
\mathcal{U}(v)(x):=\sum_{g\in G} \psi(g)\big(\chi_{\mathcal{F}}v\big)(gx).
$$
Direct calculations show us that for any $u\in \mathcal{H}_G$, $v\in H^1_0(\mathcal{F})$,
$$ u=\mathcal{U}(\chi_{\mathcal{F}}u), \;\;\mathcal{E}(\mathcal{U}(v))=\mathcal{E}_0(v), \;\;\text{ and }\;\; \cP(\mathcal{U}(v))=\cP_0(v).$$
These facts suggest that $\inf_{\mathcal{P}_{0}}\mathcal{E}_0=\inf_{\mathcal{P}_G}\mathcal{E}=c_G$,
where
$
\mathcal{P}_{0}=\{v\in H^1_0(\mathcal{F})\setminus\{0\}\st \cP_0(v)=0\}
$.

Let $u$ be the nodal solution that achieves $c_G$, then $w:=\chi_{\mathcal{F}}u\in H^1_0(\mathcal{F})$ and
$\mathcal{E}_0(w)=\inf_{\cP_0}\mathcal{E}_0=c_G$.
Note that $\abs{w}\in H^1_0(\mathcal{F})$. Since $f$ is odd, $F$ is even and we thus deduce that
$$
\mathcal{E}_0(\abs{w})=\mathcal{E}_0(w)=c_G\;\;\text{ and }\;\;
\mathcal{P}_0(\abs{w})=\mathcal{P}_0(w)=0.$$
We thus assume that $w\geq 0$ in $\mathcal{F}$. On the other hand,
$\cP(\mathcal{U}(w))=\cP_0(w)=0$ and $
\mathcal{E}(\mathcal{U}(w))=\cE_0(w)=c_G
$ imply that $\mathcal{U}(w)\in \mathcal{P}_G$ also achieves $c_G=p_G$, thus $\mathcal{U}(w)$ is  a weak solution to the Choquard equation
\eqref{eqChoquard}, see \cite[Lemma 5.1]{MVTAMS}.
By Theorem \ref{thm1.2}, we infer that $\mathcal{U}(w)\in C^2$. Therefore, $\mathcal{U}(w)$
satisfies in the classical sense that $-\Delta \mathcal{U}(w)+\mathcal{U}(w) \geq 0$ in $\mathcal{F}$,
which, together with the strong maximum principle for classical supersolutions (see \cite[Theorem 3.5]{GT1997}), yields $\mathcal{U}(w)>0$ everywhere in
${\rm{int}}(\mathcal{F})$.
From the definition of $\mathcal{U}$, we see that $u>0$ everywhere in ${\rm{int}}(\mathcal{F})$. Thus $u$ has exactly
$\abs{G}$ number of nodal domains in the whole space $\Rset^N$.

We finally complete the existence of the global minimal nodal solutions.

\noindent\textbf{Proof of Theorem \ref{thm1.4}.}
 We consider the minimization problem on the nodal solution set,
$$c_{\rm{nod}}=\inf_{u\in\cM_{\rm{nod}}}\cE(u),$$
with the constraint being defined by $\cM_{\rm{nod}}=\{u\in H^1(\Rset^N)\st u^{\pm}\neq 0, \cE'(u)=0\}$.
Clearly we have by Theorem \ref{thm1.1} that the ${\bf{A}_1}$-groundstates do exist, which leads us $c_0\leq c_{\rm{nod}}\leq c_{\bf{A}_1}<2c_0$, here $c_0$ refers to the ground state energy of the Choquard equation \eqref{eqChoquard}. Moreover, $c_0=\inf_{\cN}\cE$ where $\cN$ is the Neari manifold. This can be done due to our monotonic assumption $(F'_1)$.

Taking a minimizing sequence $(v_n)_{n\in\Nset}\subset \cM_{\rm{nod}}$ so that  $ \cE'(v_n)=0$ and $\cE(v_n)\to c_{\rm{nod}}$ as $n\to\infty$. By combining the Pohozaev identity, we obtain the boundedness of $(v_n)_{n\in\Nset}$ in $H^1(\Rset^N)$.
By the assumptions $(F'_1)$ and $(F'_2)$, there exists some constant $C>0$ such that
$$\abs{F(t)},\abs{f(t)t}\leq C(\abs{t}^r+ \abs{t}^p).$$
This, together with the facts that $\langle \cE'(v_n),v_n^\pm \rangle=0$ and the Hardy--Littlewood--Sobolev inequality, implies that
$\liminf_{n\to\infty}\|v_n^{\pm}\|>0$.
We now claim that
$$\sum_{p_1,p_2\in\{r,p\}}\liminf_{n\to\infty} \int_{\Rset^N}(I_\alpha *\abs{v_n^+}^{p_1})\abs{v_n^-}^{p_2}>0.$$
Otherwise, we deduce by the assumptions $(F'_1)$ and $(F'_2)$ again that
$$\limsup_{n\to\infty} \int_{\Rset^N}(I_\alpha *F(v_n^{\pm}))F(v_n^{\mp})\leq C\sum_{p_1,p_2\in\{r,p\}}\limsup_{n\to\infty} \int_{\Rset^N}(I_\alpha*\abs{v_n^+}^{p_1})\abs{v_n^{-}}^{p_2}=0 $$
and
$$\limsup_{n\to\infty} \int_{\Rset^N}(I_\alpha*F(v_n^\pm))f(v_n^\mp)v_n^{\mp}\leq C\sum_{p_1,p_2\in\{r,p\}}\limsup_{n\to\infty} \int_{\Rset^N}(I_\alpha *\abs{v_n^+}^{p_1})\abs{v_n^{-}}^{p_2}=0 $$
We therefrom obtain that
$$\cE(v_n)=\cE(v_n^{+})+\cE(v_n^-)+o(1),\; \text{ and }\;\langle\cE'(v_n^\pm),v_n^\pm\rangle =\langle\cE'(v_n),v_n^\pm\rangle+o(1)=o(1).$$
It then follows that there exist $s_n^\pm\in(0,+\infty)$ with
$\lim_{n\to\infty} s_n^\pm=1$ such that $s^\pm_nv_n^\pm\in \mathcal{N}$. This leads a contradiction that
\begin{align*}
  2c_0>c_{\rm{nod}}+o(1)&=\cE(v_n)=\cE(v_n^+)+\cE(v_n^-)+o(1)&\\
  &=\cE(s_n^+v_n^+)+\cE(s_n^-v_n^-)+o(1)\geq 2c_0+o(1).
\end{align*}
By adopting the proof of \cite[Lemma 3.6]{GhimentiVanSchaftingen2016}, we conclude that there exist $R>0$ and one pair $(p_1,p_2)\in\{(r,r),(r,p),(p,r),(p,p)\}$ such that
$$\liminf_{n\to\infty}\sup_{a\in\Rset^N}\int_{B_R(a)}\abs{v^+_n}^{\frac{2Np_1}{N+\alpha}}
\int_{B_R(a)}\abs{v_n^-}^{\frac{2Np_2}{N+\alpha}}>0.$$
Up to a translation and a subsequence, we can assume that $$\liminf_{n\to\infty}\int_{B_R(0)}\abs{v_n^+}^{\frac{2Np_1}{N+\alpha}}>0, \text{ and }\liminf_{n\to\infty}\int_{B_R(0)}\abs{v_n^-}^{\frac{2Np_2}{N+\alpha}}>0.$$
This means that $v_n\wto v\in H^1(\Rset^N)$ and $v^\pm\neq 0$. In particular, $\cE'(v)=0$ and $\cP(v)=0$. We then deduce by the weakly lower semicontinuity of the norm that
\begin{align*}
  c_{\rm{nod}}&\leq \cE(v)=\cE(v)-\frac{1}{N+\alpha}\cP(v)&\\
&=\frac{\alpha}{2(N+\alpha)}\int_{\Rset^N}\abs{\nabla v}^2+\frac{2+\alpha}{2(N+\alpha)}\int_{\Rset^N}\abs{v}^2 &\\
&\leq \frac{\alpha}{2(N+\alpha)}\int_{\Rset^N}\abs{\nabla v_n}^2+\frac{2+\alpha}{2(N+\alpha)}\int_{\Rset^N}\abs{v_n}^2+o(1)&\\
&=\cE(v_n)-\frac{1}{N+\alpha}\cP(v_n)+o(1)=\cE(v_n)+o(1).
\end{align*}
Hence, the infimum value $c_{\rm{nod}}$ is achieved by the sign-changing solution $v$ and then the strongly convergence follows.

\section*{Appendix}

 \begin{appendices}
\label{appendix}

\section{Finite Coxeter groups}

In this section, we shall recall some fundamental theory on Coxeter groups including the classification result and the Coxeter diagram. Most of these background knowledge are contained in the references \cite{Davis2008,Thomas2017}. We only list some definitions and results because the related proofs can be found either in \cite{Davis2008,Thomas2017} or in the previous work \cite{XiaZhang2020}. The first concept is the reflection.
\begin{definition} {\rm\cite[Definition 6.1.1]{Davis2008}}
A linear reflection on a vector space $V$ is a linear automorphism $r :V \to V$ such that $r^2 = 1$ and such the fixed subspace $H_r$ is a hyperplane.
\end{definition}
With the help of reflections, we introduce the formal definition of Coxeter groups.
\begin{definition}{\rm\cite[Definition 1.18]{Thomas2017}}\label{defcoxeter}
Let $I$ be a finite indexing set with cardinal number $\abs{I}=k$ and $\mathcal{S} =\{s_i\}_{i\in I}$.
A Coxeter matrix $M=(m_{ij})_{i, j\in I}$ on the set $\mathcal{S}$ is a  $k\times k$
symmetric matrix with entries in $\Nset\cup \{\infty\}$ such that $m_{ii}=1$  for all $i\in I$ and $m_{ij}\geq 2$ for all distinct $i, j\in I$.
The associated Coxeter group $G=G_{\mathcal{S}}$ is defined by the presentation
$$G =\langle \mathcal{S} \st s_i^2=1, \, \forall i\in I \;\text{ and }\; (s_is_j)^{m_{ij}}=1, \; \text{ for } i\neq j\in I\rangle.$$
The set $\mathcal{S}$ is a Coxeter generating set for $G$ and the cardinality $\abs{I}$ is called the rank of $G$.
\end{definition}

The following Coxeter diagram  is a useful tool in classifying the finite Coxeter groups.
\begin{definition} \rm{\cite[Definition 3.5.1]{Davis2008}}
Suppose that $M = (m_{ij})$ is a Coxeter matrix on a set $I$.
We associate to $M$ a graph $\Upsilon$ called its Coxeter graph. The vertex set
of $\Upsilon$ is $I$. Distinct vertices $i$ and $j$ are connected by an edge if and only if
$m_{ij}\geq 3$. The edge ${i,j}$ is labeled by $m_{ij}$ if $m_{ij} \geq 4$. If $m_{ij} = 3$, the edge is left
unlabeled. The graph together with the labeling of its edges is called the Coxeter diagram associated to $M$.
\end{definition}

The following proposition (see e.g., {\rm\cite[Theorem 4.1.6 and Proposition 4.1.7]{Davis2008}}) shows that
the study of finite Coxeter groups can be largely reduced to the case when the Coxeter is irreducible, which means its diagram is connected.
\begin{proposition}\label{proparabolic}
 For each $\mathcal{T}\subset\mathcal{S}$, $G_{\mathcal{T}}$ is a Coxeter group. Moreover,
 suppose $\mathcal{S}$ can be partitioned into two nonempty disjoint subsets $\mathcal{T}$ and $\mathcal{T}'$ such that $(st)^2=1$ for all $s\in \mathcal{T}$ and $t\in \mathcal{T}'$. Then
$G= G_{\mathcal{T}}\times G_{\mathcal{T}'}$.
\end{proposition}
 As a result, any finite Coxeter group $G$ is the direct products of some irreducible  Coxeter groups.
We now recall the famous classification of the finite irreducible Coxeter groups, which was first completely enumerated by Coxeter \cite{Coxeter1935}.
\begin{Theorem}\rm{\cite[Theorem 6.9.1]{Davis2008}}\label{classification}
The finite irreducible Coxeter groups and the associated Coxeter diagrams are listed in \rm{Table 1}.
\end{Theorem}
\begin{table}[htbp]
\begin{center}{}
\begin{tabular}{p{2cm}<{\centering}p{4cm}<{ }p{1cm}<{ }
p{4cm}<{ }}
${\bf{A}}_n \; (n\geq 1) $
 &
\begin{tikzpicture}
\vertex  (n1) at (0,0)  [draw]{};
  \vertex (n2) at (0.6,0) [draw]{};
  \vertex (n3) at (1.2,0) [draw]{};
  \vertex (n4) at (1.8,0) [draw]{};
  \vertex (n5) at (2.4,0) [draw]{};
  \vertex (n6) at (3.0,0) [draw]{};
\draw (n1)--(n2)  ;
  \draw (n2)--(n3);
  \draw [ dashed] (n3)--(n4);
     \draw (n4)--(n5);
      \draw (n5)--(n6);
\end{tikzpicture}
 & ${\bf{E}}_8$ &
 \begin{tikzpicture}
\vertex  (n1) at (0,0)  [draw]{};
  \vertex (n2) at (0.6,0) [draw]{};
  \vertex (a)  at (1.2,0.43) [draw] {};
  \vertex (n3) at (1.2,0) [draw]{};
  \vertex (n4) at (1.8,0) [draw]{};
  \vertex (n5) at (2.4,0) [draw]{};
\vertex (n6) at (3.0,0) [draw]{};
\vertex (n7) at (3.6,0) [draw]{};
   \draw (n1)--(n2)  ;
  \draw (n3)--(a);
   \draw (n2)--(n3);
  \draw (n3)--(n4);
     \draw (n4)--(n5);
        \draw (n5)--(n6);
         \draw (n6)--(n7);
  \end{tikzpicture}\\
\specialrule{0em}{4pt}{8pt}
${\bf{B}}_n\; (n\geq 2)$  &
\begin{tikzpicture}
\vertex   (n0) at (0,0)    [draw] { };
  \vertex (n1) at (0.6,0)  [draw] { };
  \vertex (n2) at (1.2,0)  [draw] { };
  \vertex (n3) at (1.8,0)  [draw] { };
  \vertex (n4) at (2.4,0)  [draw] { };
  \vertex (n5) at (3.0,0)  [draw] { };
  \path
    (n4) edge node[pos=0.5,above]{4} (n5);
  \draw (n0)--(n1)  ;
  \draw (n1)--(n2);
  \draw [ dashed] (n2)--(n3);
     \draw (n3)--(n4);
     \draw (n4)--(n5);
  \end{tikzpicture}
& ${\bf{H}}_3 $ &
\begin{tikzpicture}
  \vertex (n0) at (0.6,0)   [draw] { };
  \vertex  (n1) at (1.2,0)   [draw] { };
  \vertex (n2) at (1.8,0)    [draw] { };
    \path
    (n0) edge node[pos=0.5,above]{5} (n1);
   \draw (n0)--(n1)  ;
  \draw (n1)--(n2)  ;
  \end{tikzpicture}\\
\specialrule{0em}{4pt}{8pt}
 ${\centering} {\bf{D}}_n \; (n\geq 4)$   &
  \begin{tikzpicture}
\vertex   (n0) at (0,0)    [draw] { };
  \vertex (n1) at (0.6,0)  [draw] { };
  \vertex (n2) at (1.2,0)  [draw] { };
  \vertex (a)  at (1.8,0)  [draw] { };
  \vertex (n3) at (2.4,0)  [draw] { };
  \vertex (n4) at (2.4,0.43)  [draw] { };
    \vertex (n5) at (3.0,0) [draw] { };
    \draw (n0)--(n1);
  \draw (n1)--(n2);
     \draw [dashed] (n2)--(a);
     \draw  (a)--(n3);
         \draw  (n3)--(n4);
     \draw (n3)--(n5);
  \end{tikzpicture}
  & ${\bf{H}}_4$ &
  \begin{tikzpicture}
  \vertex (n0) at (0,0)  [draw] { };
  \vertex (n1) at (0.6,0)  [draw] { };
  \vertex (n2) at (1.2,0)  [draw] { };
 \vertex (n3) at (1.8,0)   [draw] { };
 \path
   (n0) edge node[pos=0.5,above]{5} (n1);
   \draw (n0)--(n1)  ;
  \draw (n1)--(n2)  ;
  \draw (n2)--(n3);
  \end{tikzpicture}\\
\specialrule{0em}{4pt}{8pt}
${\bf{E}}_6$   &
\begin{tikzpicture}
  \vertex (n1) at (0,0)  [draw] { };
  \vertex (n2) at (0.6,0)  [draw] { };
 \vertex (n3) at (1.2,0)   [draw] { };
  \vertex (n4) at (1.2,0.43)  [draw] { };
  \vertex (n5) at (1.8,0)   [draw] { };
   \vertex (n6) at (2.4,0)    [draw] { };
  \draw (n1)--(n2)  ;
  \draw (n2)--(n3);
  \draw (n3)--(n4);
     \draw (n3)--(n5);
     \draw (n5)--(n6);
  \end{tikzpicture}
   & ${\bf{F}}_4$ &
  \begin{tikzpicture}
\vertex   (n0) at (0,0)    [draw] { };
  \vertex (n1) at (0.6,0)  [draw] { };
  \vertex (n2) at (1.2,0)  [draw] { };
  \vertex (n3) at (1.8,0)  [draw] { };
   \path
    (n1) edge node[pos=0.5,above]{4} (n2);
  \draw (n0)--(n1)  ;
  \draw (n1)--(n2);
  \draw  (n2)--(n3);
  \end{tikzpicture}\\
\specialrule{0em}{4pt}{8pt}
${\bf{E}}_7$   &
 \begin{tikzpicture}
 \vertex  (n1) at (0,0)  [draw]{};
  \vertex (n2) at (0.6,0) [draw]{};
  \vertex (a)  at (1.2,0.43) [draw] {};
  \vertex (n3) at (1.2,0) [draw]{};
  \vertex (n4) at (1.8,0) [draw]{};
  \vertex (n5) at (2.4,0) [draw]{};
\vertex (n6) at (3.0,0) [draw]{};
   \draw (n1)--(n2)  ;
  \draw (n3)--(a);
   \draw (n2)--(n3);
  \draw (n3)--(n4);
     \draw (n4)--(n5);
        \draw (n5)--(n6);
  \end{tikzpicture}
  & ${\bf{I}}_2(m)$&
  \begin{tikzpicture}
  \vertex  (n1) at (0,0)  [draw]{};
  \vertex (n2)  at (0.6,0) [draw]{};
  \path
   (n1) edge node[pos=0.5,above]{$m$} (n2);
   \draw (n1)--(n2)  ;
  \end{tikzpicture}
\end{tabular}
\caption{\quad  Classification of the finite irreducible Coxeter groups}
\end{center}
\end{table}
The finite irreducible Coxeter groups consist of three one-parameter families of increasing rank ${\bf{A}}_n$, ${\bf{B}}_n$, ${\bf{D}}_n$, one one-parameter family of dimension two ${\bf{I}}_2(m)$, and six exceptional groups: ${\bf{E}}_6$, ${\bf{E}}_7$, ${\bf{E}}_8$ and ${\bf{H}}_3$, ${\bf{H}}_4$, ${\bf{F}}_4$. Each of these groups corresponds to a symmetry group of certain regular polytope (polygons) or semiregular polytope in the corresponding dimensional space.  Here we note for the diagram of ${\bf{I}}_2(m)$ that if $m=2$, there is no edge, and if $m=3$, the edge is not labeled.

The following lemmas is a direct consequence of the finite Coxeter groups, we refer to \cite[Lemma 2.4]{Thomas2017} and \cite[Lemma 2.7]{XiaZhang2020}, respectively.
\begin{lemma}\label{lemmaepimorphism}
There exists a unique epimorphism $\psi: G\to \{\pm 1\}$ induced by $\psi(s) = -1$ for
all $s\in\mathcal{S}$, so that for any $g\in G$, $\psi(g^{-1})=\psi(g)^{-1}$.
\end{lemma}
\begin{lemma} \label{normalsubgroup}
For every finite irreducible Coxeter group $W$ of rank $k$, there exists a normal subgroup $N_k\lhd W$ such that
$W= N_k\rtimes {\bf{A}}_1$.
\end{lemma}
Here we recall the semidirect product $N_k\rtimes {\bf{A}_1}$ means that $W=N_k{\bf{A}_1}$ and $N_k\cap {\bf{A}_1}=\{1\}$.

We finally recall the fundamental domain for the Coxeter group $G$, it is a closed, connected subset $\mathcal{F}\subset \mathbb R^N$ such that $Gx \cap\mathcal{F}\neq\emptyset$ for every $x\in\mathbb R^N$  and $Gx \cap \mathcal{F}=\{x\}$ for every $x$ in the interior of $\mathcal F$. It is called a strict fundamental domain if it
intersects each $G$-orbit in exactly one point. Here $Gx$ refers to the $G$-orbit of the point $x\in \mathbb R^k$.
It turns out that every Coxeter group admits a strict fundamental domain  $\mathcal{F}$ \cite[Corollary 6.6.9]{Davis2008}. In fact, under the group action $G$, for any $g\in G$, $g\mathcal{F}$  is also a fundamental domain. Without loss of generality, we can fix $\cF$ to be the infinite cone formed by $H_s$ for all $s\in \cS$.
Let $\partial^i \mathcal{F}$ denote its $k-i$ dimensional facets of the polyhedral cone $\mathcal{F}$ for $i=0,1,\cdots,k-1$, here $\partial^{0}\mathcal{F}$ refers to the interior of $\mathcal{F}$ for convenience. Let $\mathcal{G}_k$ be the collection of finite Coxeter groups of rank $k$. Also, $\mathcal{G}_0$ contains only the trivial group. We now  state an elementary results on the isotropic groups \cite[Lemma 2.8]{XiaZhang2020}.
\begin{lemma}\label{lemmainduction}
Let $k\geq 1$ and $\mathcal{F}$ be the strict fundamental domain of $G\in \mathcal{G}_k$. Then for any $q\in \partial^{i} \mathcal{F}$ with $i=0,\cdots,k-1$,
$S_q\in \mathcal{G}_{i}$.
\end{lemma}
We recall that the isotropic group $S_q$ for $q\in \mathcal{F}$ is the subgroup of $G$ that satisfies $S_q=\{g\in G\st gq=q\}$.
We close this appendix by citing the representation result for the finite Coxeter groups, which is owe to Tits
\cite[Theorem 3.1]{Thomas2017}.
\begin{lemma}
Let $G$ be a finite Coxeter group with rank $k$. There exists a faithful representation
$$\rho: G\to GL(\mathbb R^k ).$$
\end{lemma}
A faithful representation means that group homomorphism $\rho$ is injective. Namely, the realized space for the group $G\in \mathcal{G}_k$ at least is $\mathbb R^k$, so that the restriction $k\leq N$ in Theorem \ref{thm1.1} is to ensure the well realization of the action of Coxeter group $G$.
\end{appendices}

\bigskip

\noindent{\bf Acknowledgement.}
This paper is supported by National Natural Science Foundation of China (Grant No. 11771324, 11811540026 and 11901456)
and by the Natural Science Foundation of Shaanxi Province (Grant no. 2020JQ--120).

\bigskip
\bibliographystyle{plain}

\end{document}